\documentclass{amsart}

\usepackage{amssymb,amsmath,amsfonts,euscript,amscd,mathrsfs,amsthm,graphicx}

\numberwithin{equation}{section}
\numberwithin{figure}{section}

\newtheorem{lemma}{Lemma}[section]
\newtheorem{theorem}[lemma]{Theorem}
\newtheorem{corollary}[lemma]{Corollary}
\newtheorem{proposition}[lemma]{Proposition}

\newtheorem{definition}{Definition}[section]
\theoremstyle{remark}
\newtheorem{remark}{Remark}[section]

\newcommand{\D}{\mathbb{D}}
\newcommand{\half}{\mathbb{H}}
\newcommand{\R}{\mathbb{R}}
\newcommand{\N}{\mathbb{N}}
\newcommand{\Z}{\mathbb{Z}}

\newcommand{\prob}[2]{\mathbf{P}^{#1} \left \{ #2 \right \}}
\newcommand{\abs}[1]{\left | #1  \right |}
\newcommand{\paren}[1]{\left ( #1 \right )}
\newcommand{\bracket}[1]{\left [ #1 \right ]}
\newcommand{\set}[1]{\left \{ #1 \right \}}
\newcommand{\ev}[2]{\mathbf{E}^{#1} \left[ #2 \right ] }

\DeclareMathOperator{\hm}{hm}
\DeclareMathOperator{\imag}{Im}

\newcommand{\C}{\mathbb{C}}

\begin{document}
\author{Shawn Drenning}
\title {Excursion Reflected Brownian Motion}
\date{\today}

\begin{abstract}
Excursion reflected Brownian motion (ERBM) is a strong Markov process defined in a finitely connected domain $D \subset \C$ that behaves like a Brownian motion away from the boundary of $D$ and picks a point according to harmonic measure from infinity to reflect from every time it hits a boundary component.  We give a construction of ERBM using its conformal invariance and develop the basic theory of its harmonic functions.  One important reason for studying ERBM is the hope that it will be a useful tool in the study of SLE in multiply connected domains.  To this end, we develop the basic theory of the Poisson kernel and Green's function for ERBM and show how it can be used to construct conformal maps into certain classes of multiply connected domains.
\end{abstract}
\maketitle

\section{Introduction}

\subsection{Motivation and Results}

Roughly speaking, if $D \subset \C$ is a domain with $n$ ``holes,'' \emph{excursion reflected Brownian motion} (ERBM) is a strong Markov process that has the distribution of a Brownian motion away from $\partial D$ and picks a point according to \emph{harmonic measure from $\infty$} to reflect from every time it hits $\partial D$.  To understand the behavior of ERBM, we consider the case that $D=\C \backslash \D$. Intuitively, ERBM in $\C \backslash \D$ can be constructed by taking a reflected Brownian motion and rotating each excursion from $\D$ by an angle chosen uniformly from $[0,2\pi)$.  ERBM has what Walsh (\cite{MR509476}, pg. 37) has called a ``roundhouse singularity'' in a neighborhood of $\D$.  That is, in any neighborhood of a time that it hits $\partial \D$, it will hit $\partial \D$ uncountably many times and jump randomly from point to point on $\partial \D$.  An important property of ERBM (that we will use as part of our definition) is that it is conformally invariant.  This will be clear once we more precisely define what it means to ``pick a point according to harmonic measure from $\infty$ to reflect from.''

An important reason to consider ERBM is that it arises naturally when studying conformal maps into certain classes of multiply connected domains.  A classical theorem of complex analysis states that if $D\subset \C$ is an $n$-connected domain and $w\in \partial D$, then there is a conformal map $f=u+iv$ from $D$ onto the upper half-plane with $n$ horizontal line segments removed satisfying $f\paren{w}=\infty$.  If $n=0$, then $v$ is a positive harmonic function that vanishes on $\partial D$ except at $w$. It is well-known that this characterizes $v$ as being a real multiple of the Poisson kernel for Brownian motion $H_D\paren{\cdot,w}$.  If $n>0$, then it is well-known \cite{MR1344449} that $v$ is a harmonic function that is constant on each boundary component of $D$ and that for any smooth Jordan curve $\eta \subset D$ we have
\begin{equation}\label{conCond}
\int_{\eta} \frac{d}{dn}~v\paren{z}~\abs{dz}=0,
\end{equation}
where $n$ is the outward pointing unit normal.  Let $w\in A_0, A_1, \ldots, A_n$ be the connected components of the boundary of $D$. An easy calculation \cite{MR2247843} shows that \eqref{conCond} holds exactly if for every $A_i$, $i\neq 0$, and smooth Jordan curve $\eta\subset D$ with $A_i$ (and no other boundary component) in its interior, we have
\begin{equation}\label{IntroERHar}
v\paren{A_i}=\int_{\eta} v\paren{z} \frac{H_{\partial U}\paren{A_i,z}}{\mathcal{E}_{U}\paren{A_i,\eta}}~\abs{dz},
\end{equation}
where $U$ is the region bounded by $\partial A_i$ and $\eta$, $H_{\partial U}$ is the boundary Poisson kernel, and $\mathcal{E}_{U}$ is excursion measure (see Section \ref{sectBG}).  This suggests that $v$ is a real multiple of the Poisson kernel of a stochastic process with state space $D \cup \set{A_0,\ldots, A_n}$ that has the distribution of a Brownian motion in $D$ and started at $A_i$, a density for the distribution of where it first hits $\eta$ is $\frac{H_{\partial U}\paren{A_i,z}}{\mathcal{E}_{U}\paren{A_i,\eta}}$.  We essentially define ERBM to be such a process.

The existence of a process similar to ERBM follows from more general work of Fukushima and Tanaka in \cite{MR2139028}.  Their work uses the theory of Dirichlet forms and does not take advantage of the conformal invariance of ERBM.  An alternative construction making explicit use of the conformal invariance of ERBM was proposed by Lawler in \cite{MR2247843}.  He proposed that ERBM could be defined in any domain with ``one hole'' by first constructing the process in $\C \backslash \D$ using excursion theory and then defining it in any domain conformally equivalent to $\C \backslash \D$ via conformal invariance.  To define ERBM in a domain with ``$n$ holes,'' (or more generally, countably many holes) multiple copies of the process defined in a domain with ``one hole'' can be pieced together.  We take this basic approach and give a new construction of ERBM.

A function is \emph{ER-harmonic} if it satisfies the mean value property with respect to ERBM.  More precisely, a function $u$ is ER-harmonic if it is harmonic on $D$ and \eqref{IntroERHar} holds.  Two important ER-harmonic functions are the Poisson kernel $H_D^{ER}\paren{z,w}$ and Green's function $G_D^{ER}\paren{z,w}$ for ERBM. In order to define these functions, it is necessary to choose at least one boundary component of $D$ at which to kill the ERBM.  Once this is done, the definitions and many of the properties of the Poisson kernel and Green's function for ERBM are similar to those for usual Brownian motion.  The Poisson kernel for ERBM was first considered by Lawler in \cite{MR2247843} as a way of understanding a classical theorem \cite{MR510197} of complex analysis stating that any $n$-connected domain $D \subset \C$ is conformally equivalent to a domain obtained by removing $n$ horizontal line segments from $\half$.  He sketched a proof showing that the imaginary part of any such map is equal to a real multiple of the Poisson kernel for ERBM.  We give a complete proof here.  Furthermore, we use the Green's function for ERBM to prove two other classical conformal mapping theorems.

\subsection{Outline of the Paper}

Section \ref{sectBG} sets notation and contains some necessary background material.  In Section \ref{sectERBM} we define and construct ERBM in finitely connected domains. First, we construct the process in $\C \backslash \D$ by explicitly defining a semigroup for ERBM in terms of the semigroups for Brownian motion and reflected Brownian motion and then using general theory to show that there actually is a strong Markov process with this semigroup.  Finally, we check that the strong Markov process we obtain satisfies our definition of ERBM.  Our construction is motivated by a similar construction of Walsh's Brownian motion in \cite{MR1022917}.  Once we have ERBM in $\C \backslash \D$, we define ERBM in any domain conformally equivalent to $\C \backslash \D$ via conformal invariance. In Section \ref{chapERBMFC} we construct ERBM in finitely connected domains by computing what its infinitesimal generator would be if it existed and then using general theory to show that there actually is a Feller-Dynkin process with that infinitesimal generator.  ERBM in a finitely connected domain induces a discrete time Markov chain on the connected components of the boundary of $D$, which we discuss in Section \ref{chapMC}.  This chain was observed by Lawler in \cite{MR2247843} and appears implicitly in classical work on conformal mapping of multiply connected domains.  We conclude the section with a brief discussion of the harmonic functions associated with ERBM, which we call \emph{ER-harmonic} functions.  We prove a maximal principle for ER-harmonic functions and show how ERBM can be used to construct ER-harmonic functions.


We discuss the Poisson kernel and Green's function for ERBM in Sections \ref{sectPK} and \ref{sectGF} respectively.  We prove some of their basic properties and show how they can be used to construct conformal maps into certain classes of finitely connected domains.




I would like to thank my thesis advisor Greg Lawler for suggesting this line of research and for many useful conversations pertaining to it.

\section{Background} \label{sectBG}
\subsection{Some Notation}

We denote the unit disk in $\C$ centered at the origin by $\D$ and the upper half-plane by $\half$. We let $\mathcal{Y}_n$ consist of all subdomains of $\C$ with $n$ ``holes.''  More precisely, let $\mathcal{Y}_n$ consist of all connected domains of the form 
\[D=\C \backslash \bracket{A_0 \cup A_1 \cup \cdots \cup A_n},\]
where $A_0, A_1, \ldots ,A_n$ are closed disjoint subsets of $\C$ such that $A_i$ is simply connected, bounded, and larger than a single point for $1 \leq i \leq n$ (we allow $A_0$ to be empty) and $\C \backslash A_0$ is simply connected.  We will often think of $A_0\cup \set{\infty}$ as being a single point at infinity (the point we need to add to make $D\cup \set{A_1,\ldots, A_n}$ with its quotient topology compact).  We denote $\displaystyle \bigcup_{i=0}^{\infty} \mathcal{Y}_i$ by $\mathcal{Y}$.


We denote the open annulus centered at $0$ with inner radius $r$ and outer radius $R$ by $A_{r,R}$ and the open ball of radius $r$ centered at $z$ by $B_r\paren{z}$.

If $E$ is a locally compact Hausdorff space and $E_{\partial}=E\cup \set{\partial}$ is the one-point compactification of $E$, we denote by $C_0\paren{E}$ the set of all continuous real-valued functions on $E$ that vanish at $\partial$.  If $D \in \mathcal{Y}$, we denote by $C^{\infty}\paren{D}$ the set of all infinitely differentiable functions on $D$.

We will use $c$ to denote a real constant that is allowed to change from one line to the next.  We write $f\paren{z} \sim g\paren{z}$ as $z \rightarrow a$ if $\lim_{z \rightarrow a}\frac{f\paren{z}}{g\paren{z}}=1$. 

\subsection{Poisson Kernel for Brownian Motion}
Let $D \in \mathcal{Y}$ and let $\tau_D$ be the first time that a Brownian motion $B_t$ leaves $D$.  If $\partial D$ has at least one regular point for Brownian motion, then for each $z \in D$, the distribution of $B_{\tau_D}$ defines a measure $\hm_D\paren{z,\cdot}$ on $\partial D$ (with the $\sigma$-algebra generated by Borel subsets of $\partial D$) called \emph{harmonic measure in $D$ from $z$}.  We say $\partial D$ is \emph{locally analytic} at $w \in \partial D$ if $\partial D$ is an analytic curve in a neighborhood of $w$.  If $\partial D$ is locally analytic at $w$, then in a neighborhood of $w$, $\hm_D\paren{z,\cdot}$ is absolutely continuous with respect to arc length and the density of $\hm_D\paren{z,\cdot}$ at $w$ with respect to arc length is called the \emph{Poisson kernel for Brownian motion} and is denoted $H_D\paren{z,w}$.  If $w$ is a two-sided boundary point, we should really think of it as being two distinct boundary points, $w^+$ and $w^-$.  In such cases, by abuse of notation, we will sometimes write $H_D\paren{z,w}$ when we should consider $H_D\paren{z,w^+}$ and $H_D\paren{z,w^-}$ separately.

Harmonic measure is conformally invariant.  That is, if $f: D \rightarrow D'$ is a conformal map, then 
\[\hm_D\paren{z,V}=\hm_{D'}\paren{f\paren{z},f\paren{V}}.\]
Using this, we see that if $\partial D$ is locally analytic at $w$ and $\partial D'$ is locally analytic at $f\paren{w}$, then
\begin{equation} \label{PKCI}
H_{D'}\paren{f\paren{z},f\paren{w}}=\abs{f'\paren{w}}^{-1}H_D\paren{z,w}.
\end{equation}
It is well-known that
\begin{equation}\label{PKhalf}
H_{\half}\paren{x+iy,x'}=\frac{1}{\pi} \frac{y}{\paren{x-x'}^2+y^2}.
\end{equation}
A useful fact \cite{MR2247843} that we will use is that if $D_2 \subset D_1$ and $\partial D_1$ and $\partial D_2$ agree and are locally analytic in a neighborhood of $w \in \partial D_1$, then
\begin{equation} \label{twoDomainPK}
H_{D_2}\paren{z,w}=H_{D_1}\paren{z,w}-\ev{z}{H_{D_1}\paren{B_{\tau_{D_2}},w}}.
\end{equation}

The function $H_D\paren{\cdot,w}$ can be characterized up to a positive multiplicative constant as the unique positive harmonic function on $D$ that is ``equal to'' the Dirac delta function at $w$ on $\partial D$. 

\begin{proposition}\label{noSubLinHar}
Let $D \in \mathcal{Y}$ be such that $\partial D$ is locally analytic at $w\in \partial D$.  Then $H_D\paren{\cdot,w}$ is up to a real constant multiple the unique positive harmonic function on $D$ that satisfies $H_D\paren{z,w} \rightarrow 0$ as $z \rightarrow w'$ for any $w' \in \partial D$ not equal to $w$.
\end{proposition}

\subsection{Excursion Measure}
Let $D \in \mathcal{Y}$.  If $\partial D$ is locally analytic at $w$, then the \emph{boundary Poisson kernel} is defined by
\[H_{\partial D}\paren{w,z}=\frac{d}{dn} ~H_D\paren{w,z},\]
where $n$ is the inward pointing normal at $w$.  If $w$ is a two-sided boundary point, we will adopt a convention similar to the one we adopted for $H_D\paren{w,z}$ when $z$ is two-sided. If $f$ is a conformal map and $\partial f\paren{D}$ is locally analytic at $f\paren{w}$ and $f\paren{z}$, then
\begin{equation}\label{BdPKCI}
H_{\partial D}\paren{w,z}=\abs{f'\paren{w}}\abs{f'\paren{z}}H_{\partial f\paren{D}}\paren{f\paren{w},f\paren{z}}.
\end{equation}

The definition of excursion reflected Brownian motion uses excursion measure.  Excursion measure is sometimes defined as a measure on paths between two boundary points of $D$.  Since we will only be interested in the norm of that measure, the definition we give of excursion measure is the norm of excursion measure as defined elsewhere (\cite{MR2129588}, \cite{MR2247843}).  

\begin{definition}
Suppose $D \subset \C$ is a domain with locally analytic boundary and $V$ and $V'$ are disjoint arcs in $\partial D$.  Then
\begin{equation*}
\mathcal{E}_D\paren{V,V'} :=\int_{V}\int_{V'} H_{\partial D}\paren{z,w} \abs{dz}\abs{dw}
\end{equation*}
is called \emph{excursion measure}.  
\end{definition}

Using \eqref{BdPKCI}, we can check that $\mathcal{E}_D$ is conformally invariant.  This allows us to define $\mathcal{E}_D\paren{V,V'}$ even if $D$ does not have locally analytic boundary.  We will often write $\mathcal{E}_D\paren{A,V}$ for $\mathcal{E}_D\paren{\partial A,V}$ and $H_{\partial D}\paren{A,z}$ as shorthand for the quantity $\int_{\partial A} H_{\partial D}\paren{z,w} \abs{dz}$.  Using \eqref{BdPKCI}, we see that if $f: D \rightarrow D'$ is a conformal map, then 
\[H_{\partial D}\paren{A,z}=H_{\partial f\paren{D}}\paren{f\paren{A},f\paren{z}}\abs{f'\paren{z}}.\]
As a result, it is possible to define $H_{\partial D}\paren{A,z}$ even if $A$ does not have locally analytic boundary.

\subsection{Green's Function for Brownian Motion}

In what follows, let $D \in \mathcal{Y}$ be such that $\partial D$ has at least one regular point for Brownian motion.  In this setting, it is possible to define a (a.s. finite) Green's function for Brownian motion $G_D\paren{z,w}$ (see, for instance, \cite{MR2129588}).  By convention, we scale $G_D\paren{z,\cdot}$ so that it is the density for the occupation time of Brownian motion.  As a result, what we mean by $G_D$ may differ by a factor of $\pi$ from what appears elsewhere.

It is well-known that $G_D\paren{z,w}=G_D\paren{w,z}$ and that $G_D\paren{z,\cdot}$ can be characterized as the unique harmonic function on $D \backslash \set{z}$ such that $G_D\paren{z,w} \rightarrow 0$ as $w \rightarrow \partial D$ and 
\begin{equation} \label{GreenAsztow}
G_D\paren{z,w}=\frac{-\log \abs{z-w}}{\pi}+O\paren{1},
\end{equation}
as $z \rightarrow w$.  Another property of $G_D\paren{z,w}$ is that it is conformally invariant.  That is, if $f:D \rightarrow D'$ is a conformal map, then $G_{f\paren{D}}\paren{f\paren{z},f\paren{w}}=G_D\paren{z,w}$.  Finally, it is well-known that
\begin{equation} \label{GFUniDisk}
G_{r\D}\paren{0,z}=-\frac{\log r - \log \abs{z}}{\pi}.
\end{equation}

\section{Excursion Reflected Brownian Motion} \label{sectERBM} 

\subsection{Definition}
\label{ChapERBMDef}
We start this section by giving a precise definition of excursion reflected Brownian motion in $D \in \mathcal{Y}$.  We will see that for any $D \in \mathcal{Y}$ there is a unique process satisfying the conditions of our definition.

The Jordan curve theorem says that any Jordan curve $\eta$ separates $\C$ into exactly two connected components.  We will call the bounded connected component the \emph{interior} of $\eta$ and the unbounded connected component the \emph{exterior} of $\eta$. If $A \subset \C$ is in the interior of $\eta$, we will say $\eta$ \emph{surrounds} $A$.

\begin{definition}
\label{characterizationERBM}
Let $E=D\cup \set{A_1, \ldots, A_n}$ be equipped with the quotient topology and let $E_{\partial}=E \cup \set{A_0}$ be the one-point compactification of $E$. A strong Markov process $B^{ER}_{D}$ with state space $E_{\partial}$ is called an \emph{excursion reflected Brownian motion} (ERBM) if it satisfies the following properties.

\begin{enumerate}
\item $B_D^{ER}$ has continuous sample paths.

\item If we start the process at $z \in D$ and let
\[T=\inf \set{t:  B_{D}^{ER}\paren{t} \in \partial D},\]
then for $0\leq t \leq T$, $B_{D}^{ER}\paren{t}$ is a Brownian motion in $D$ killed at $\partial D$.

\item Let $\eta_1, \ldots, \eta_n$ be pairwise disjoint smooth Jordan curves in $D$ such that $\eta_i$ surrounds $A_i$ and does not surround $A_j$ for $j\neq i$.  If
\[\sigma=\inf \set{t:B_{D}^{ER}\paren{t} \in \eta_i},\]
then $B_{D}^{ER}\paren{\sigma}$ has the distribution of $\frac{\mathcal{E}_{U_i}\paren{A_i,\cdot}}{\mathcal{E}_{U_i}\paren{A_i,\eta_i}}$, where $U_i$ is the region bounded by $\partial A_i$ and $\eta_i$.

\item  $B_{D}^{ER}$ is conformally invariant (this will be made more precise in Proposition \ref{ERBMconformalinvariance}) and the radial part of $B_{\C \backslash \D}^{ER}$ has the same distribution as the radial part of a reflected Brownian motion in $\C \backslash \D$.
\end{enumerate}
\end{definition}

We think of $A_0$ as being a ``coffin'' state; once the process is in $A_0$, it can never leave. We will often refer to ERBM in $D$ or $E$ when we really mean the process with the enlarged state space $E_{\partial}$.

\subsection{Excursion Reflected Brownian Motion in $\C \backslash \D$}
\label{ChapERBMCD}
The first step in constructing ERBM is to construct it in $E=\C \backslash \D \cup \set{\overline{\D}}$.  We will mimic the construction of Walsh's Brownian motion given in \cite{MR1022917}.  The idea of the construction is that if a process exists that satisfies Definition \ref{characterizationERBM}, we can determine what its semigroup must be.  Once we know what its semigroup must be, we use general theory to show that there actually is a process with that semigroup.  Finally, once we have the process, we check that it actually satisfies Definition \ref{characterizationERBM}.  For the remainder of this section, we will use polar coordinates to specify points in $E$.

We will build the semigroup for ERBM using the semigroup for reflected Brownian motion in $\C\backslash \D$ and Brownian motion in $\C\backslash \D$.  There is much in the literature on reflected Brownian motion and it is possible to define it in very general domains.  However, in $\C \backslash \D$ it is possible to give a simple construction.   Let $B_1$ and $B_2$ be independent one-dimensional Brownian motions and define reflected Brownian motion in $\half$ to be the process $B_1+i\abs{B_2}$.  We can then define reflected Brownian motion in $\C \backslash \D$ to be the image of reflected Brownian motion in $\half$ under the map $z \mapsto e^{-i z}$ with the appropriate time change.  It is not hard to verify that this definition agrees with other definitions in the literature and that the resulting process is a Feller-Dynkin process (see \cite{MR1796539} for the definition and basic properties of Feller-Dynkin processes).
 
\begin{proposition}\label{semigroupERBMThm}
Let $T_t^+$ be the semigroup for reflected Brownian motion in $\C\backslash \D$ and $T_t^0$ be the semigroup for Brownian motion in $\C\backslash \D$.  For $f \in C_0\paren{E}$, let
\begin{equation}
P_t f\paren{r,\theta}=T^+_t \overline{f}\paren{r,\theta}+T_t^0 \paren{f-\overline{f}}\paren{r,\theta},
\label{semigroupERBM}
\end{equation}
where $\displaystyle \overline{f}\paren{r,\theta}=\frac{1}{2\pi} \int_0^{2\pi} f\paren{r,\theta}d\theta.$  If there is a stochastic process $B_{\C \backslash \D}^{ER}$ taking values in $E$ that satisfies Definition \ref{characterizationERBM}, then its semigroup is $P_t.$
\end{proposition}

\begin{proof}
Assume we have a process $X_t$ taking values in $E$ that satisfies Definition \ref{characterizationERBM} and a filtration $\paren{\Omega,\mathcal{F}_t}$ to which $X_t$ is adapted to. Let $\tau$ be the first time $X_t$ hits $\overline{\D}$.  Modifying the filtration if necessary, the D\'{e}but theorem says that the first time $\tau$ that $X_t$ hits $\overline{\D}$ is a stopping time.  Finally, let
\[A_t=\set{\omega \in \Omega: \tau\leq t}.\]
Definition \ref{characterizationERBM} implies that $X_t$ has the distribution of a Brownian motion for $t \leq \tau$ and that on $A_t$ the angular part of $X_t$ is uniformly distributed and the radial part is that of a reflected Brownian motion.  Combining these facts, we have that if $f\in C_0\paren{E}$, then
\begin{align*}\displaystyle
P_t f\paren{x}&= \ev{x}{f\paren{X_t}}\\
&= \ev{x}{\mathbf{1}_{A_t}f\paren{X_t}}+\ev{x}{\mathbf{1}_{A_t^c}f\paren{X_t}} \\
&= T_t^+ \overline{f}\paren{x}-\ev{x}{\mathbf{1}_{\Omega \backslash A_t}\overline{f}\paren{X_t}}+T_t^0 f\paren{x}-\ev{x}{\mathbf{1}_{A_t}\overline{f}\paren{\overline{\D}}}\\
&= T^+_t \overline{f}\paren{r,\theta}+T_t^0 \paren{f-\overline{f}}\paren{r,\theta}.
\end{align*}
\end{proof}

\begin{proposition}
$P_t$ is a Feller-Dynkin semigroup on $C_0\paren{E}$.  
\end{proposition}

\begin{proof}
Using the fact that $T_t^+$ and $T_t^0$ are Feller-Dynkin semigroups, this is a straightforward exercise. See \cite{DreThesis} for a complete proof.

\end{proof}

Using (say) Theorem III.7.1 of \cite{MR1796539}, given any measure $\mu$ on $E$, we can define a unique Feller-Dynkin process
 \begin{equation} \label{ERBMdef}
B^{ER}_{\C\backslash \D}:=\paren{\Omega,\mathscr{F},\{\mathscr{F}_t:t\geq 0\},\{B^{ER}_{\C\backslash \D}(t):t\geq 0\},\mathbf{P}^{\mu}}
\end{equation}
with semigroup $P_t$.  Furthermore, the filtration $\mathscr{F}_t$ is independent of the measure $\mu$, $B^{ER}_{\C\backslash \D}$ has the strong Markov property with respect to $\mathscr{F}_t$, and the sample paths of $B^{ER}_{\C \backslash \D}$ are c\`{a}dl\`{a}g.  We denote the angular and radial parts of $B^{ER}_{\C\backslash \D}$ at time $t$ by $\theta_t$ and $R_t$ respectively.

Next we check that the process $B^{ER}_{\C\backslash \D}$ defined in \eqref{ERBMdef} satisfies Definition \ref{characterizationERBM}.
\begin{proposition}
$B^{ER}_{\C \backslash \D}$ has the distribution of a Brownian motion up until the first time it hits $\partial D$.
\label{ERBMlikeBM}
\end{proposition}

\begin{proof}
This follows immediately from \eqref{semigroupERBM}.
\end{proof}

\begin{proposition} \label{radialpartERBM}
$R_t$ has the same distribution as the radial part of a reflected Brownian motion in $\C \backslash \D$.
\end{proposition}

\begin{proof}
We mimic the proof of Lemma 2.2 in \cite{MR1022917}.  Let $g \in C_0\paren{\left[1,\infty \right)}$ and define $f \in C_0\paren{E}$ by $f\paren{r,\theta}=g\paren{r}.$  Observe that $\overline{f}=f$.  If $S$ is any $\mathscr{F}_t$-stopping time, then
\begin{align*}
\ev{\mu}{g\paren{R_{S+t}}|\mathscr{F}_S}&=\ev{\mu}{f\paren{R_{S+t},\theta_{S+t}}|\mathscr{F}_S}\\
&= P_t f\paren{R_S,\theta_S}\\
&= T_t^+ \overline{f}\paren{R_S,\theta_S}+T_t^0\paren{f-\overline{f}}\paren{R_S,\theta_S}\\
&= T_t^+ f \paren{R_S,\theta_S}\\
&= R_t^+ g\paren{R_S},
\end{align*}
where $R_t^+$ is the semi-group for the radial part of reflected Brownian motion in $\C \backslash \D$. The result follows.
\end{proof}

\begin{proposition} \label{excursiondistERBM}
Let $\eta$ be a smooth Jordan curve surrounding $\overline{\D}$, $U$ be the region bounded by $\eta$ and $\partial \D$, and $\tau$ be the first time $B^{ER}_{\C \backslash \D}$ hits $\eta$. If $V$ is an arc in $\eta$, then
\begin{equation*}
\alpha:= \prob{\overline{\D}}{B^{ER}_{\C \backslash \D}(\tau) \in V}=\frac{\mathcal{E}_U\paren{\D,V}}{\mathcal{E}_U\paren{\D,\eta}}.
\end{equation*}
\end{proposition}

\begin{proof}
Let $C_{\epsilon}$ be the circle of radius $1+\epsilon$ centered at the origin.  Since it is clear from \eqref{semigroupERBM} that $B_{\C\backslash \D}^{ER}$ is rotationally invariant, the result follows in the case that $\eta=C_{\epsilon}$.

Let $p\paren{z}$ be the probability that a Brownian motion started at $z$ exits $U$ on $\eta$.  For small enough $\epsilon$, $C_{\epsilon}$ is in the interior of $\eta$.  For such an $\epsilon$, using the strong Markov property for ERBM and Proposition \ref{ERBMlikeBM}, we see that
\begin{align*}
2\pi \paren{1+\epsilon} \alpha &= \int_{C_{\epsilon}}\bracket{\int_V H_U\paren{z,w} \abs{dw}+\paren{ 1-p\paren{z}}\alpha}\abs{dz} \\
&=2\pi \paren{1+\epsilon} \alpha+\int_{C_{\epsilon}} \bracket{\int_V H_U\paren{z,w} \abs{dw}-p\paren{z}\alpha}\abs{dz}\\
&=2\pi \paren{1+\epsilon} \alpha+\int_{C_{\epsilon}} \bracket{\int_V H_U\paren{z,w} \abs{dw}-\alpha\int_{\eta} H_U\paren{z,w}\abs{dw}}\abs{dz}.
\end{align*}
As a result, for small enough $\epsilon$, we have
\[\displaystyle\alpha=\frac{\int_{C_{\epsilon}}\int_V H_U\paren{z,w} \abs{dw}\abs{dz}}{\int_{C_{\epsilon}}\int_{\eta} H_U\paren{z,w}\abs{dw}\abs{dz}}.\]
Since the derivative of $H_U\paren{\cdot,w}$ is bounded in a neighborhood of $\D$ (we can extend $H_U\paren{\cdot,w}$ to a function harmonic in a neighborhood of $\partial \D$), using the mean value theorem and dominated convergence, we see that
\begin{align*}
\alpha &=\lim_{\epsilon \rightarrow 0} \frac{\int_{C_{\epsilon}}\int_V \frac{H_U\paren{z,w}}{\epsilon} \abs{dw}\abs{dz}}{\int_{C_{\epsilon}}\int_{\eta} \frac{H_U\paren{z,w} }{\epsilon}\abs{dw}\abs{dz}}\\
&= \frac{\int_{\partial \D}\int_V H_{\partial U}\paren{z,w} \abs{dw}\abs{dz}}{\int_{\partial \D}\int_{\eta} H_{\partial U}\paren{z,w} \abs{dw}\abs{dz}}\\
&= \frac{\mathcal{E}_U\paren{\D,V}}{\mathcal{E}_U\paren{\D,\eta}}.
\end{align*}
\end{proof}

\begin{proposition}
There is a unique process stochastic process with state space $E=\C\backslash D\cup \set{\D}$ satisfying Definition \ref{characterizationERBM}.
\end{proposition}

\begin{proof}
The uniqueness statement follows from Proposition \ref{semigroupERBM}.  Propositions \ref{ERBMlikeBM}, \ref{radialpartERBM}, and \ref{excursiondistERBM} combine to show that the process defined in \eqref{ERBMdef} has the strong Markov property and satisfies (2), (3), and (4) of Definition \ref{characterizationERBM}.  By construction, $B_{\C \backslash \D}^{ER}$ has right continuous paths.  Since Brownian motion has a continuous modification and a right continuous process that has a continuous modification is already continuous, it follows that $B_{\C \backslash \D}^{ER}$ is continuous in $\C \backslash \D$.  Similarly, we can use the fact that the radial part of reflected Brownian motion has a continuous modification to show $B_{\C \backslash \D}^{ER}$ is continuous at $\D$.
\end{proof}

\begin{remark}
If $\mathcal{A}$ is the infinitesimal generator for $B_{\C\backslash \D}^{ER}$ and $f \in C^{\infty}\paren{\C\backslash \D}$ is in the domain of $\mathcal{A}$, then
\begin{equation}\displaystyle
\mathcal{A}f\paren{x}=
\begin{cases}
1/2 \Delta f\paren{x} & \mbox{if } x \in \C\backslash \D\\
\mathcal{A} f\paren{x}=\frac{1}{2\pi}\int_0^{2\pi} \frac{d}{dn}~f\paren{e^{i\theta}}~d\theta & \mbox{if } x=\D,
\end{cases}
\end{equation}
where $n$ is the outward pointing unit normal.  By appropriately modifying $\mathcal{A}$, we can obtain processes similar to ERBM with slightly different behavior at the boundary.  These processes will not have the connections to conformal mapping that ERBM enjoys though.

\end{remark}

\subsection{Excursion Reflected Brownian Motion in Conformal Annuli}
\label{chapERBMCA}
Let $A$ be any compact, connected subset of $\C$ larger than a single point and
\[f: \C \backslash \D \rightarrow \C \backslash A\]
be a conformal map sending $\infty$ to $\infty$.  It is a straightforward exercise to verify that $f$ is unique up to an initial rotation.  Let $\sigma_t$ be the $\mathscr{F}_t$ stopping time given by
\[\int_0^{\sigma_t} \abs{f'\paren{B_{\C\backslash \D}^{ER}\paren{s}}}^2 ds=t\]
and define
\[B_{\C\backslash A}^{ER}\paren{t}=f\paren{B_{\C \backslash D}^{ER}\paren{\sigma_t}}\]
and $\tilde{\mathscr{F}}_t=\mathscr{F}_{\sigma_t}.$
We define ERBM in $\C \backslash A$ to be the process
\begin{equation*}
B^{ER}_{\C \backslash A}:=\paren{\Omega,\mathscr{F},\set{\tilde{\mathscr{F}}_t: t\geq 0}, \set{B^{ER}_{\C\backslash A}},\set{\mathbf{P}^x}}.
\end{equation*}
Since $B_{\C \backslash D}^{ER}$ is rotationally invariant and $f$ is unique up to an initial rotation, it is clear that the distribution of $B_{\C\backslash A}^{ER}$ does not depend on $f$.  It is well-known that such a time change preserves the strong Markov property (see the discussion on pg. 277 of \cite{MR1796539}).  Using the fact that $B^{ER}_{\C \backslash A}$ behaves locally like a Browian motion, it is an easy exercise to check that $B^{ER}_{\C \backslash A}$ is a Feller-Dynkin process. 
To ensure that $B_{\C\backslash A}^{ER}\paren{t}$ exists for all $t<\infty$, we need to verify that
\begin{equation}
\int_0^{\infty} \abs{f'\paren{B_{\C\backslash \D}^{ER}\paren{s}}}^2 ds=\infty~~\rm{a.s.}.
\label{ERBMalltime}
\end{equation}
In order for $B_{\C\backslash A}^{ER}\paren{t}$ not to have a limit as $t \rightarrow \infty$, we need to verify that for all $t<\infty$,
\begin{equation}
\int_0^{t} \abs{f'\paren{B_{\C\backslash \D}^{ER}\paren{s}}}^2 ds<\infty~~\rm{a.s.}.
\label{ERBMnolimit}
\end{equation}
We temporarily put these considerations aside.

\begin{proposition}\label{CIERBMAnn}
Suppose $f: \C \backslash \D \rightarrow D_1$ and $g : D_1 \rightarrow D_2$ are conformal maps.  Then the process
\[B^{ER}_{D_2}\paren{t}=B^{ER}_{D_1}\paren{\sigma_t},\]
where
\[\int_0^{\sigma_t} \abs{g'\paren{B^{ER}_{D_1}\paren{s}}}^2 ds=t\]
is an ERBM in $D_2$. \label{ERBMconformalinvariance}
\end{proposition}

\begin{proof}
Let $\sigma_r$ satisfy
\[\int_0^{\sigma_r} \abs{g'\paren{f\paren{B^{ER}_{\C \backslash \D}\paren{s}}}f'\paren{B^{ER}_{\C \backslash \D}\paren{s}}}^2 ds=r\]
and define a map $T:\bracket{0,\sigma_r} \rightarrow \left[0, \infty \right)$ by
\begin{equation}
t \mapsto \int_0^t \abs{f'\paren{B^{ER}_{\C \backslash \D}\paren{s}}}^2 ds.
\end{equation}
It is straightforward to verify that $T$ is a bijection (we use \eqref{ERBMnolimit} here) onto $\bracket{0,T\paren{\sigma_r}}$ with derivative $\abs{f'\paren{B^{ER}_{\C \backslash \D}\paren{s}}}^2$.  Using the change of variables formula, we have
\begin{align*}
r&=\int_0^{\sigma_r} \abs{g'\paren{f\paren{B^{ER}_{\C \backslash \D}\paren{s}}}f'\paren{B^{ER}_{\C \backslash \D}\paren{s}}}^2 ds\\
&=\int_0^{\sigma_r} \abs{g'\paren{B^{ER}_{D_1}\paren{T\paren{s}}}f'\paren{B^{ER}_{\C \backslash \D}\paren{s}}}^2 ds\\
&=\int_0^{T\paren{\sigma_r}} \abs{g'\paren{B^{ER}_{D_1}\paren{s}}}^2 ds.
\end{align*}
As a result, $B^{ER}_{D_2}\paren{r}=g\paren{B^{ER}_{D_1}\paren{T\paren{\sigma_r}}}=g\paren{f\paren{B^{ER}_{\C \backslash \D}\paren{\sigma_r}}}$ and thus, the process in $D_2$ defined by $g$ is the same as the process defined by $g\circ f$.  The result follows.
\end{proof}

Proposition \ref{CIERBMAnn} is what we mean when we say ERBM is conformally invariant. 

\begin{proposition}
There is a unique stochastic process $B_{\C \backslash A}^{ER}$ with state space $E=\C\backslash A \cup \set{A}$ satisfying Definition \ref{characterizationERBM}.
\end{proposition}

\begin{proof}
The uniqueness follows from (4) of Definition \ref{characterizationERBM} and Proposition \ref{CIERBMAnn}. Since $B^{ER}_{\C \backslash \D}$ has continuous sample paths, it is clear that $B^{ER}_{\C \backslash A}$ does as well.  The fact that $B_{\C \backslash A}^{ER}$ satisfies (2) and (3) of Definition \ref{characterizationERBM} follows from the conformal invariance of Brownian motion and excursion measure respectively.
\end{proof}

If $A_0$ is a closed subset of $\C \backslash A$ it makes sense to discuss ERBM in $\C \backslash A$ killed at $A_0$.  Most often we will do this when $A_0$ is a simple, closed curve $\eta$ surrounding $A$ and refer to the corresponding process as ERBM in $D$, where $D$ is the region bounded by $\eta$ and $\partial A$.  It is well-known that stopping a process the first time it hits a closed set preserves the Markov property and, in fact, it is not hard to verify that in our case the Feller property is preserved as well. 

\subsection{Excursion Reflected Brownian Motion in Finitely Connected Domains} \label{ERBMGeneralDef}
\label{chapERBMFC}
Let $D \in \mathcal{Y}_n$ and $E$ be as in Definition \ref{characterizationERBM}.  Intuitively, we can define ERBM in $D$ killed at $A_0$ pathwise to be a Brownian motion up until the first time it hits an $A_i$, then be ERBM in $\C\backslash A_i$ until it hits an $A_j$ with $j \neq i$, then be ERBM in $\C \backslash A_j$ and so on. We can make this rigorous by looking at infinitesimal generators.  If there were a process satisfying Definition \ref{characterizationERBM}, it is easy to check (using the fact that it behaves locally like a Brownian motion) that it would be a Feller-Dynkin process.  Furthermore, since $B_D^{ER}$ has continuous paths (and the infinitesimal generator for Brownian motion is the closure of $1/2 \Delta$), if $\mathcal{A}$ were the infinitesimal generator of $B_{D}^{ER}$ and $f\in C^{\infty}\paren{D}$ was in the domain of $\mathcal{A}$, then we would have
\begin{equation} \label{FCGenerator}
\mathcal{A}f\paren{x}=
\begin{cases}
1/2 \Delta f\paren{x} & \mbox{if } x \in D \\
0 & \mbox{if } x=A_0 \\
\mathcal{A}_i f\paren{x} & \mbox{if } x=A_i,~ 1\leq i \leq n,
\end{cases}
\end{equation}
where $\mathcal{A}_i$ is the infinitesimal generator for $B_{\C\backslash A_i}^{ER}$.  Define an operator $\mathcal{A}: D\paren{\mathcal{A}} \rightarrow C_0\paren{E}$ pointwise by \eqref{FCGenerator}, where $D\paren{\mathcal{A}}$ consists of all $f \in C_0\paren{E}$ such that $\mathcal{A}f \in C_0\paren{E}$. Using the following topological fact, it is easy to check that $D\paren{\mathcal{A}}$ and the image of $I-\mathcal{A}$ are both dense in $C_0\paren{E}$.

\begin{lemma}
Let $U_1, \ldots, U_m$ be an open cover of $E$, $S \subset C_0\paren{E}$ be a linear space, and for each $1 \leq i \leq m$, $S_i \subset S$ be a subspace of functions with support in $U_i$.  If the natural inclusion of $S_i$ into $C_0\paren{U_i}$ is dense for each $i$, then $S$ is dense in $C_0\paren{E}$.
\end{lemma} 

As a result, the Hille-Yosida theorem \cite{MR838085} implies that $\mathcal{A}$ is the infinitesimal generator of a strongly continuous contraction semigroup.  
 We define $B_{D}^{ER}$ to be the corresponding Feller-Dynkin process.  It is easy to check that $B_{D}^{ER}$ defined in this way is the unique strong Markov process with state space $E$ satisfying Definition \ref{characterizationERBM}.
 
\begin{remark}
For ease of notation, we have focused on finitely connected domains, but the construction we have given works just as well for countably connected domains so long as we can find countably many Jordan curves $\eta_1, \eta_2, \ldots$ such that each boundary component is in the interior of exactly one $\eta_i$. 
\end{remark}

\subsection{A Markov Chain Associated with ERBM} \label{MarkovChainERBM}
\label{chapMC}
Let $D\in \mathcal{Y}_n$, $\eta_i$, for $1\leq i \leq n$, be as in Definition \ref{characterizationERBM} and $h_i$, for $1 \leq i \leq n$, be the unique bounded harmonic function on $D$ that is equal to $1$ on $\partial A_i$ and $0$ on $\partial A_j$ for $j \neq i$ (note that $h_j\paren{z}$ is the probability that a Brownian motion started at $z$ exits $D$ at $A_j$).  ERBM on $D$ induces a discrete time Markov chain $X$ with state space $\set{A_0, \ldots, A_n}$ (see \cite{MR2247843} pg. 37).  The probability that the chain moves from $A_i$ to $A_j$ is equal to the probability that $A_j$ is the first boundary component of $D$ that $B_D^{ER}$ started at $A_i$ hits after the first time it hits $\eta_i$.  That is, the chain has transition probabilities $p_{00}=1$ and
\begin{equation*}
p_{i j}=\int_{\eta_i} h_j\paren{z} \frac{H_{\partial U_i}\paren{A_i,z}}{\mathcal{E}_{U_i}\paren{A_i, \eta_i}} \abs{dz},
\end{equation*}
for $i\neq 0$.  This Markov chain is not entirely satisfactory since it is highly dependent on the particular choice of $\eta_1, \ldots ,\eta_n$.  By erasing all of the loops from $X$ we obtain a Markov chain $Y$ with transition probabilities $q_{0 0}=1$, $q_{i i}=0$ for $i>0$, and
\begin{equation*}
q_{i j}=\frac{p_{i j}}{1-p_{ii}},
\end{equation*}
for $i \neq j$.  It is not hard to see that $Y$'s transition probabilities are independent of the choice of $\eta_1, \ldots ,\eta_n$.  Since $q_{j0}>0$ for all $1 \leq j \leq n$, the eigenvalues of the transition matrix $\mathbf{Q}$, for $Y$ restricted to $A_1, \ldots, A_n$, have absolute value strictly less than one and, using standard results from Markov chain theory, we have that the Green's matrix
\begin{equation}\label{GreenMatrixFormula}
\mathbf{I}+\mathbf{Q}+\mathbf{Q}^2+\cdots + \mathbf{Q}^n+ \cdots=\paren{\mathbf{I}-\mathbf{Q}}^{-1}.
\end{equation}
is well-defined.

\subsection{Excursion Reflected Harmonic Functions}
\begin{definition}\label{ERHarmonicDef}
A function
\[v:E \rightarrow \R\]
is called \emph{ER-harmonic} if it satisfies
\begin{enumerate}
\item $v$ is continuous on $E$ and is harmonic when restricted to $D$
\item For $1 \leq i \leq n$, if $\eta$ is a Jordan curve surrounding $A_i$, then
\begin{equation} \label{ERBMMeanvalue}
v\paren{A_i}=\int_{\eta} v\paren{z} \frac{H_{\partial U_i}\paren{A_i,z}}{\mathcal{E}_{U_i}\paren{A_i, \eta}} \abs{dz},
\end{equation}
where $U_i$ is the region bounded by $\eta$ and $\partial A_i.$
\end{enumerate}
\end{definition}

If it is clear what is meant, we will sometimes speak of the ER-harmonicity of a function with domain $D$ rather than $E$.  By an ER-harmonic function on $D-\set{z}$ or $D-\set{A_i}$ we mean a function that satisfies Definition \ref{ERHarmonicDef} except that (2) is not necessarily satisfied for curves surrounding $z$ and $A_i$ respectively.

The following is a useful criterion for a function to be ER-harmonic.

\begin{lemma}\label{ERHarmonicCondition}
Let $\eta$ be a smooth Jordan curve surrounding $A_j$ and not surrounding $A_i$ for $i\neq j$. Then for any harmonic function $v$ on $D$ we have
\begin{equation*}
\int_{\eta} v\paren{z} H_{\partial U_j}\paren{A_j,z}\abs{dz}=v\paren{A_j}\mathcal{E}_{U_j}\paren{A_j,\eta}+ \int_{\eta} \frac{d}{dn}~ v\paren{z}\abs{dz},
\end{equation*}
where $U_j$ is the region bounded by $\eta$ and $\partial A_j$.  In particular, if $v$ is continuous on $E$, then $v$ is ER-harmonic if and only if for each $i$ there is an $\eta_i$ surrounding $A_i$ with
\[\int_{\eta_i} \frac{d}{dn}~ v\paren{z}\abs{dz}=0.\]
\end{lemma}

\begin{proof}
See \cite{MR2247843} pg. 17.
\end{proof}

As with harmonic functions, if we specify suitable boundary conditions, there is a unique ER-harmonic function with these boundary conditions.  The key to proving this uniqueness is a maximal principle for ER-harmonic functions.

\begin{lemma}[Maximal principle for ER-harmonic functions] \label{ERBMMaximalprinciple}
Let $v: E \cup \partial A_0 \rightarrow \R$ be a bounded, continuous function that is ER-harmonic when restricted to $E$.  Then
\begin{enumerate}
\item  The maximum of value of $v$ is equal to the maximum value of $v$ restricted to $\partial A_0$.
\item  If there is a $z \in E$ such that $v$ attains its maximum at $z$, then $v$ is constant.
\end{enumerate}
\end{lemma}

\begin{proof}
It is clear that (2) implies (1), so it is enough to prove (2).  Let $z$ be a point where $v$ attains its maximum.  If $z \in D$, then by the strong maximal principle for harmonic functions \cite{MR1625845}, $v$ is constant.  If $z = A_i$, then using \eqref{ERBMMeanvalue} it is clear there is some $z' \in D$ where $v$ also attains its maximum and thus, $v$ is constant.
\end{proof}

\begin{proposition} \label{ERHarmonicFromERBM}
Suppose that $\partial A_0$ has at least one regular point for Brownian motion and let $F: \partial A_0 \rightarrow \R$ be a bounded, measurable function.  Define
\[v: \overline{D} \rightarrow \R\]
by
\[v \paren{z}=\ev{z}{F\paren{B^{ER}_D\paren{\tau_D}}},\]
where $\tau_D$ is the first time $B_D^{ER}$ hits $A_0$.  Then $v$ is a bounded $ER$-harmonic function on $D$ that is continuous at all regular points of $\partial A_0$ at which $F$ is continuous.  Furthermore, if every point of $\partial A_0$ is regular and $F$ is continuous, then $v$ is the unique ER-harmonic function that is equal to $F$ on $\partial A_0$.
\end{proposition}

\begin{proof}
It is clear from the fact that $F$ is bounded that $v$ is also bounded.  The proof that $v$ is harmonic and continuous at the regular points of $A_0$ at which $F$ is continuous is similar to the proof of the corresponding result for Brownian motion (see \cite{MR2604525}).  The fact that \eqref{ERBMMeanvalue} holds follows from the strong Markov property for ERBM and (3) of Definition \ref{characterizationERBM}.  The uniqueness statement follows from a straightforward application of Lemma \ref{ERBMMaximalprinciple}.
\end{proof}

\section{The Poisson Kernel for ERBM}\label{sectPK}
\subsection{Definition and Basic Properties}
Throughout this section let $D \in \mathcal{Y}_n$ be such that $A_0 \neq \emptyset$ and let
\begin{equation*}
\tau_D=\inf\set{t\in \R^+ : B_D^{ER}\paren{t} \in \partial A_0 }.
\end{equation*}
The distribution of $B_D^{ER}\paren{\tau_D}$ defines a measure $\hm_D^{ER}\paren{z, \cdot}$ on $\partial A_0$ (with the $\sigma$-algebra generated by Borel subsets of $\partial A_0$) that we call \emph{ER-harmonic measure in $D$ from $z$}.  Using the analogous result for harmonic measure and the construction of ERBM, it is easy to check that if $\partial D$ is locally analytic at $w$, then $\hm_D^{ER}\paren{z, \cdot}$ is absolutely continuous with respect to arc length in a neighborhood of $w$.  The density of $\hm_D^{ER}\paren{z, \cdot}$ at $w$ with respect to arc length is called the \emph{Poisson kernel} for ERBM and is denoted $H_D^{ER}\paren{z,w}$.

If $\gamma:\paren{-\delta,\delta}\rightarrow \partial A_0,$ $\gamma\paren{0}=w$ is an analytic curve, then we can explicitly define a version of $H_D^{ER}\paren{z,w}$ by 
\begin{equation}\label{explicitPoisson}
\displaystyle H_D^{ER}\paren{z,w}=\lim_{\epsilon \rightarrow 0} \frac{\hm_D^{ER}\paren{z,\gamma \paren{ -\epsilon,\epsilon}} } {\int_{-\epsilon}^{\epsilon}\abs{\gamma'\paren{x}}~dx}.
\end{equation}
It is clear that this definition is independent of $\gamma$.  In what follows, when we refer to $H_D^{ER}\paren{z,w}$, we will mean the version given by \eqref{explicitPoisson}.

An analog of \eqref{PKCI} holds for $H_D^{ER}\paren{z,w}$.

\begin{proposition} \label{ERPKCI}
If $f: D \rightarrow D'$ is a conformal transformation such that $\partial D$ is locally analytic at $w$ and $\partial D'$ is locally analytic at $f\paren{w}$, then
\begin{equation*}
H_{D'}^{ER}\paren{f \paren{z},f \paren{w}}=\abs{f'\paren{w}}^{-1}H_D^{ER}\paren{z,w}.
\end{equation*}
\end{proposition}

\begin{proof}
Since ERBM is conformally invariant, $\hm_D^{ER}\paren{z, \cdot}$ is conformally invariant.  Combining this with the change of variables formula, the result follows.
\end{proof}

Recall that $h_i\paren{z}$ is the unique bounded harmonic function on $D$ that is $1$ on $\partial A_i$ and $0$ on $\partial A_j$ for $j \neq i$.  If $V$ is a Borel subset of $\partial A_0$, then using the strong Markov property for ERBM, we see that
\begin{equation*} 
\hm_D^{ER}\paren{z,V}=\hm_D\paren{z,V}+\sum_{i=1}^n h_i\paren{z}\hm_D^{ER}\paren{A_i,V}.
\end{equation*}
Combining this with \eqref{explicitPoisson}, we see that
\begin{equation}\label{ERBMPKD}
H_D^{ER}\paren{z,w}=H_D\paren{z,w}+\sum_{i=1}^n h_i\paren{z} H_D^{ER}\paren{A_i,w}.
\end{equation}
Using \eqref{ERBMPKD}, it is sometimes possible to explicitly compute $H_D^{ER}\paren{z,w}$.

\begin{proposition} \label{AnnulusERBMPK}
If $r>1$ and $A_{e^{-r},1}$ is the annulus with $\partial A_0=\partial \D$ and $\partial A_1=\partial \paren{e^{-r}\D}$, then
\[ H_{A_{e^{-r},1}}^{ER}\paren{e^{i\paren{x+iy}},1}=\frac{-\log\abs{z}}{2\pi r}+\sum_{k \in \Z} \frac{\sin\paren{\frac{\pi y}{r}}}{2r\bracket{\cosh \paren{\frac{\pi \paren{x+2\pi k}}{r}}-\cos\paren{\frac{\pi y}{r}}}}.\]
\end{proposition}

\begin{proof}
See \cite{DreThesis}.
\end{proof}


\subsection{Conformal Mapping Using $H_D^{ER}(\cdot,w)$} \label{sectPKconMap}
%
%
%
%

\label{chapPKconMap}
Recall that a domain is called a chordal standard domain if it obtained by removing a finite number of horizontal line segments from the upper half-plane.  It is a classical theorem of complex analysis \cite{MR510197} that every $D \in \mathcal{Y}_n$ is conformally equivalent to a chordal standard domain.  Furthermore, this equivalence is unique up to a scaling and real translation.  This section is devoted to using $H_D^{ER}\paren{\cdot,w}$ to give a new proof of this fact.  Our proof is based on the sketch of a proof given in \cite{MR2247843}.  In what follows, we assume that $\partial A_0$ is locally analytic at $w \in \partial A_0$.  

There is an analytic characterization of $H_D^{ER}\paren{\cdot, w}.$ 

\begin{proposition} \label{ERBMPoissonKernelCharacterization}
$H_D^{ER}\paren{\cdot,w}$ is up to a real constant multiple the unique positive ER-harmonic function that satisfies $H_D^{ER}\paren{z,w} \rightarrow 0$ as $z \rightarrow w'$ for any $w' \in \partial A_0$ not equal to $w$.
\end{proposition}

\begin{proof}
Using \eqref{ERBMPKD}, we see that $H_D^{ER}\paren{\cdot,w}$ is harmonic on $D$.  If $V$ is a Borel subset of $\partial A_0$, then it follows from the strong Markov property for ERBM and (3) of Definition \ref{characterizationERBM} that $\hm_D^{ER}\paren{\cdot,V}$ is ER-harmonic.  As a result, if $\gamma$ is as in \eqref{explicitPoisson} and $\eta$ and $U_i$ are as in Definition \ref{ERHarmonicDef}, then
\begin{align*}
H_D^{ER}\paren{A_i,w} &= \lim_{\epsilon \rightarrow 0} \frac{\hm_D^{ER}\paren{A_i,\gamma\paren{-\epsilon,\epsilon}}} {\int_{-\epsilon}^{\epsilon}\abs{\gamma'\paren{x}}~dx}\\
&=  \lim_{\epsilon \rightarrow 0} \int_{\eta} \frac{\hm_D^{ER}\paren{z,\gamma\paren{-\epsilon,\epsilon}}}{\int_{-\epsilon}^{\epsilon}\abs{\gamma'\paren{x}}~dx} \cdot \frac{H_{\partial U_i}\paren{A_i,z}}{\mathcal{E}_{U_i}\paren{A_i,\eta}} ~\abs{dz}\\
&=\int_{\eta_i} H_D^{ER}\paren{z,w} \cdot \frac{H_{\partial U_i}\paren{A_i,z}}{\mathcal{E}_{U_i}\paren{A_i,\eta}} ~\abs{dz},
\end{align*}  
where the last equality follows from the Harnack inequality and dominated convergence.  This proves that $H_D^{ER}\paren{\cdot,w}$ is ER-harmonic.  It is clear from \eqref{ERBMPKD} and Proposition \ref{noSubLinHar} that $H_D^{ER}\paren{\cdot,w}$ has the required asymptotics at $\partial A_0$.

Suppose $f$ is another positive ER-harmonic function that satisfies $f\paren{z} \rightarrow 0$ as $z \rightarrow w'$ for any $w' \in \partial A_0$ not equal to $w$.  The function
\[g\paren{z}:=f\paren{z}-\sum_{i=1}^n h_i\paren{z}f\paren{A_i}\]
is a harmonic function with $g\paren{A_i}=0$ for each $1 \leq i \leq n$ that has the same boundary conditions as $f$ at $ \partial A_0$.  It follows from Proposition \ref{noSubLinHar} that there is a $c>0$ such that $g\paren{z}=c H_D\paren{z,w}$.  As a result, \eqref{ERBMPKD} implies $f\paren{z}-c H^{ER}_D\paren{z,w}$ is a bounded ER-harmonic function that is $0$ on $A_0$ and thus, by the maximal principle for ER-harmonic functions, $f\paren{z}=c H^{ER}_D\paren{z,w}$ for all $z \in D$.
\end{proof}

A useful fact about $H^{ER}_D\paren{\cdot,w}$ that is not always true for $H_D\paren{\cdot,w}$ when $D$ is multiply connected is that $H^{ER}_D\paren{\cdot,w}$ has no critical values (that is, its derivative always has full rank).  This result will be crucial later when we prove that the level sets of $H^{ER}_D\paren{\cdot,w}$ are Jordan curves.  First we need two lemmas.

\begin{lemma} \label{PKLSConnected}
For each $r>0$, the set
\[V_r=\set{z\in D:H^{ER}_D\paren{z,w} \leq r}\]
is connected.
\end{lemma}

\begin{proof}
Using Proposition \ref{ERPKCI}, we may assume that $\C \backslash A_0=\D$.  Let $U_r$ consist of all $z \in D$ such that $H^{ER}_D\paren{z,w}<r$ and that there is a path contained in $V_r$ from $z$ to $\partial A_0-\set{w}$.  If $z_1, z_2 \in U_r$, then there are curves $\gamma_i$, for $i=1,2$, in $V_r$ connecting $z_i$ to $\partial \D$.  By staying ``close'' to $\partial \D$, we can find a path $\gamma_3$ in $V_r$ connecting $\gamma_1$ and $\gamma_2$.  It follows that $U_r$ is path connected.  It is straightforward to verify that $U_r$ is open and that $H^{ER}_D\paren{z,w}=r$ on $\partial U_r \cap D$.  

Observe that $H_D^{ER}\paren{\cdot,w}$ restricted to $D \backslash \overline{U_r}$ is an ER-harmonic function with boundary value always greater than or equal to $r$.  As a result, $H_D^{ER}\paren{z,w} \geq r$ if $z \notin U_r$ (we use the fact that if $u$ is an ER-harmonic function on $D'$, then $u\paren{z}\geq \ev{z}{u\paren{B_{D'}^{ER}\paren{\tau_{D'}}}}$, see Remark 2.12 of \cite{MR2129588}). 
If there were a $z \notin \overline{U_r}$ with $H^{ER}_D\paren{z,w}=r$, then, by the maximal principle for ER-harmonic functions, we would have that $H^{ER}_D\paren{u,w}=r$ for all $u$ in a neighborhood of $z$. Since a harmonic function on a connected domain that is constant on a non-empty open set is constant everywhere, this would imply $H^{ER}_D\paren{z,w}=r$ for all $z \in D$.  This is a contradiction, so we conclude that $V_r=\overline{U_r}$.  Since $U_r$ is connected, it follows that $V_r$ is as well.
\end{proof}

\begin{lemma} \label{locallypoly}
If $f: D \rightarrow \C$ is a holomorphic function such that 
\[f\paren{z}=\paren{z-a}^n g\paren{z}\] 
for a holomorphic function $g$ that is non-zero in a neighborhood of $a$, then there exists a conformal map $h$ defined in a neighborhood of $a$ such that $f \circ h^{-1}\paren{z}=\paren{z-a}^n$.
\end{lemma}

\begin{proof}
Since $g\paren{z} \neq 0$ in a neighborhood of $a$, we can define a branch of $\sqrt[n]{g\paren{z}}$ in a neighborhood of $a$.  A straightforward application of the argument principle shows that $h\paren{z}:=\paren{z-a} \sqrt[n]{g\paren{z}}+a$ is injective in a neighborhood of $a$ and thus is a conformal map onto its image.  It is easy to check that $f \circ h^{-1}\paren{z}=\paren{z-a}^n$.
\end{proof}

\begin{proposition} \label{PKFullRank}
For all $z \in D$, the derivative of $H^{ER}_D\paren{\cdot,w}$ at $z$ has full rank.
\end{proposition}
\begin{proof}
Suppose there were an $a\in D$ such that the derivative of $v\paren{\cdot}:=H^{ER}_D\paren{\cdot,w}$ at $a$ were zero.  Let $r=v\paren{a}$ and $U_r$ be as in Lemma \ref{PKLSConnected}.  Since $v$ is harmonic, we can find a holomorphic function $f$ defined in a neighborhood of $a$ with imaginary part equal to $v$.  Let $n$ be the order of the zero of $f\paren{z}-f\paren{a}$ at $a$.  Lemma \ref{locallypoly} implies that there is a conformal map $h$ defined on a neighborhood of $a$ such that $f \circ h^{-1}\paren{z} = \paren{z-a}^n+f\paren{a}$.  

The set of points where $v$ is equal to $r$ is equal to the image under $h^{-1}$ of the zero set of $\imag \bracket{\paren{z-a}^n}$.  As a result, for small enough $\epsilon$, the set of points where $v$ is equal to $r$ separates $B_{\epsilon}\paren{a}$ into $2n$ ordered connected components which alternate between being subsets of $U_r$ and $D\backslash U_r$.  

Let $x$ and $y$ be points in distinct connected components of $U_r \cap B_{\epsilon}\paren{a}$.  Since $U_r$ is an open, connected set, we can find a path in $U_r$ connecting $x$ and $y$.  By assumption, this path cannot be contained in $B_{\epsilon}\paren{a}$ and, therefore, a subset of it is a path $\gamma_1:[s,t] \rightarrow D\backslash B_{\epsilon}\paren{a}$ connecting the connected components of $U_r \cap B_{\epsilon}\paren{a}$ containing $x$ and $y$ respectively.  Let $\gamma_2$ be a path in $U_r \cap B_{\epsilon}\paren{a}$ connecting $a$ to $\gamma_1\paren{s}$, $\gamma_3$ be a path in $U_r \cap B_{\epsilon}\paren{a}$ connecting $\gamma_1\paren{t}$ to $a$, and $\gamma$ be the concatenation $\gamma_2$, $\gamma_1$, and $\gamma_3$.  Observe that $\gamma$ is a Jordan curve and that $v\leq r$ on $\gamma$.  As a result, Proposition \ref{ERHarmonicFromERBM} implies that $v \leq r$ on the interior of $\gamma$.  However, this is a contradiction because the interior of $\gamma$ contains one of the arcs of $\partial B_{\epsilon}\paren{a}$ connecting $\gamma\paren{s}$ and $\gamma\paren{t}$ and both these arcs contain points where $v\paren{z}>r$.  The result follows.
\end{proof}

Next we prove that the level sets of $H_D^{ER}\paren{\cdot,w}$ are Jordan curves.

\begin{proposition} \label{PKLevelCurveJordan}
If $r$ is a positive real number, then
\[\gamma_r:=\set{w}\cup \set{z:H^{ER}_D\paren{z,w}=r}\]
is a Jordan curve.\footnote{If $r=H_D^{ER}\paren{A_i,w}$ for some $i$, then in order for this to make sense, we have to work in the space $E$, not $D$.} Furthermore, $\gamma_r$ separates $D$ into two connected components,
\[ \set{z: H^{ER}_D\paren{z,w}<r } \text{ and } \set{z: H^{ER}_D\paren{z,w}>r }.\]
\end{proposition}

\begin{proof}
We will only consider the case where $r \neq H^{ER}_D\paren{A_i,w}$ for any $i$.  The other case is similar.  Using Proposition \ref{ERPKCI}, we may assume that $\partial A_0= \R$ and $w=0$.  Proposition \ref{PKFullRank} and basic facts from differential topology imply
\[K_r:=\set{z \in D :H^{ER}_D\paren{z,w}=r}\]
is a one-dimensional smooth real manifold.  As a result, each connected component of $K_r$ is diffeomorphic to either a circle or $\R$.  Using Proposition \ref{ERHarmonicFromERBM}, we see that the latter is not possible as it would imply $H_D^{ER}\paren{\cdot,w}$ is constant on an open subset of $D$.  As a result, it is not difficult to see that adding the point $w$ to any connected component of $K_r$ yields a Jordan curve $\gamma_r$.


Using \eqref{ERBMPKD}, \eqref{twoDomainPK}, and the Gambler's ruin estimate, we can show that $H_D\paren{z,0}\sim H_{\half}\paren{z,0}$ as $z \rightarrow 0$.  Combining this with \eqref{PKhalf}, we see that
\[\lim_{x \rightarrow 0} H_D^{ER}\paren{x + ax^2 i,0}=\frac{a}{\pi},\]
for any $a>0$.  It follows that either for all $a>\pi r$ there is an $\delta_a>0$ such that the interior of $\gamma_r$ contains $t+t^2i$ for all $\abs{t}<\delta_a$ or for all $a>\pi r$ there is an $\delta_a>0$ such that the exterior of $\gamma_r$ contains $t+t^2i$ for all $\abs{t}<\delta_a.$  In the latter case, it is easy to see that $H_D^{ER}\paren{\cdot,w}$ is bounded on the interior of $\gamma_r$, and thus, using Proposition \ref{ERHarmonicFromERBM}, is equal to $r$ on the interior of $\gamma_r$.  Since this is not possible, the former case must hold. In this case, if $K_r$ has two distinct Jordan curves $\gamma_r$ and $\gamma_r'$ in it, then it is clear one of them must be contained in the interior of the other.  An argument similar to the one in the proof of Lemma \ref{PKLSConnected} shows that this is not possible.  As a result, $K_r$ has exactly one connected component and the result follows.
\end{proof}

We now have all of the tools necessary to show that $H_D^{ER}\paren{\cdot,w}$ is the imaginary part of a conformal map onto a chordal standard domain.

\begin{theorem} \label{PKConformalMapTheorem}
Let $D \in \mathcal{Y}_n$ and suppose $\partial A_0$ is a smooth Jordan curve (in the topology of $E$) such that there is no Jordan curve in $D$ with $A_0$ in its interior.  If $w \in \partial A_0$, then there is a $D' \in \mathcal{CY}_n$ and conformal map $f: D \rightarrow D'$ such that $f\paren{w}=\infty$ and $\imag\bracket{f\paren{z}}=H_D^{ER}\paren{z,w}$.  Furthermore, if $g$ is another such map, then there are real constants $r,x$ such that $g=r f+x$.
\end{theorem}

\begin{proof}
Using Lemma \ref{ERHarmonicCondition} and (say) Proposition 13.3.5 of \cite{MR1344449}, we see that a harmonic function $h$ that is continuous on $E$ is the imaginary part of a holomorphic function if and only if it is ER-harmonic.  It follows that if $D'\in \mathcal{CY}_n$ and $f:D \rightarrow D'$ is a conformal map with $f\paren{w}=\infty$, then the imaginary part of $f$ is a positive ER-harmonic function such that $f\paren{z} \rightarrow 0$ as $z \rightarrow w'$ for any $w'\neq w$.  By Proposition \ref{ERBMPoissonKernelCharacterization}, this implies that the imaginary part of $f$ is a real constant multiple of $H_D^{ER}\paren{\cdot,w}$.  Combining this with the fact the imaginary part of a holomorphic function determines the real part up to a real additive constant, we obtain the uniqueness statement.

As noted above, $v\paren{\cdot}:=H_D^{ER}\paren{\cdot,w}$ is the imaginary part of a holomorphic function $f=u+iv$.  Furthermore, $u$ is defined up to a real additive constant by
\begin{equation}\label{holomorphicFromImaginary}
u\paren{z}= u\paren{z_0} + \int_{\gamma} \frac{d}{dn}~v\paren{z} \abs{dz},
\end{equation}
where $\gamma$ is a smooth curve connecting $z_0$ and $z$ and the normal derivative is chosen with the correct sign. To complete the proof, we need to show that $f$ is injective and $f\paren{D} \in \mathcal{CY}_n$.

Proposition \ref{PKLevelCurveJordan} implies that the sign of $\frac{d}{dn}~v\paren{z}$ is constant on $\gamma_r$.  As a result, $u\paren{\gamma_r\paren{t}}$ is increasing (when the appropriate parametrization of $\gamma_r$ is chosen).  In fact, $u\paren{\gamma_r\paren{t}}$ is strictly increasing since otherwise $f$ would be constant on a segment of a curve (and hence everywhere constant).  It follows that if $r \neq v\paren{A_i}$ for any $i$, then $f$ is injective on $\gamma_r$.  If $r=v\paren{A_i}$ and $z$ and $w$ are two points on $\gamma_r$, it is not hard to see that we can still find a curve connecting $v$ and $w$ on which the sign of $\frac{d}{dn} v\paren{z}$ is constant.  Arguing as before, it follows that $f$ is injective on $\gamma_r$.

Let $w_{\epsilon} \in \partial A_0$ be distance $\epsilon$ away from $w$ in the counterclockwise direction and $n_{\epsilon}$ be the inward pointing normal at $w_{\epsilon}$.  Using \eqref{PKhalf} and \eqref{PKCI}, we can check that $\frac{d}{dn_{\epsilon}} v\paren{w_{\epsilon}} \sim \frac{1}{\pi \epsilon^2}$ as $\epsilon \rightarrow 0$.  Since $\gamma_r$ is tangent to $\partial A_0$ at $w$, it follows that
\[\frac{d}{dn}~v\paren{\gamma_r\paren{t_{\epsilon}}}\sim \frac{1}{\pi \epsilon^2},~~~\epsilon \rightarrow 0,\]
where $t_{\epsilon}$ is such that $\gamma_r\paren{t_{\epsilon}}$ is distance $\epsilon$ from $w$.  It follows from \eqref{holomorphicFromImaginary} that $\abs{u\paren{z}} \rightarrow \infty$ as $z$ approaches $w$ along $\gamma_r$ and hence, $f\paren{D} \in \mathcal{CY}_n$.
\end{proof}

\begin{remark}
It is not hard to check that a version of Theorem \ref{PKConformalMapTheorem} holds for (suitable) countably connected domains with essentially the same proof.  
\end{remark}

\section{The Green's Function for ERBM} \label{sectGF}

\subsection{Definition and Basic Properties}

Throughout this section, let $D \in \mathcal{Y}$ be such that it is possible to define a Green's function $G_D\paren{z,w}$ for Brownian motion.  Recall that we normalize $G_D\paren{z,\cdot}$ so that it is a density for the expected amount of time a Brownian motion started at $z$ spends in a set before exiting $D$.

\begin{definition}
\[G^{ER}_{D}\paren{z,\cdot}:E \rightarrow \R\]
is a \emph{Green's function for ERBM} if for any Borel subset $V \subset D$
\begin{equation}\label{GrDef}
\mu_z\paren{V}:=\ev{z}{\int_0^{\tau_{D}} \mathbf{1}_{V}\paren{B_{D}^{ER}\paren{t}} dt}=\int_V G^{ER}_{D}\paren{z,w} dw,
\end{equation}
where $\tau_{D}=\inf \set{t: B_{D}^{ER}\paren{t} \in \partial A_0}$.
\end{definition}

Using the definition of ERBM and the analogous fact for Brownian motion, it is easy to prove that the probability that ERBM started at $z$ is in a set of Lebesgue measure zero at some fixed time is $0$.  Combining this fact with Fubini's theorem, we see that if $V$ has Lebesgue measure zero, then
\[\mu_z\paren{V}=\int_0^{\tau_{D}} \prob{z}{B_{D}^{ER}\paren{t} \in V} dt =0.\]
As a result, we can define $G^{ER}_{D}\paren{z,\cdot}$ as a Radon-Nikodym derivative.  Furthermore, we have 
\begin{equation}\label{RNGreenERBM}
\displaystyle G_D^{ER}\paren{z,w}=\lim_{\epsilon \rightarrow 0} \frac{\mu_z\paren{B\paren{w,\epsilon}}}{m\paren{B\paren{w,\epsilon}}}
\end{equation}
is a Green's function for ERBM, where $m$ is Lebesgue measure.  A priori, there is no reason the Green's function as defined cannot be infinite on a set of positive measure.  This potential issue will be resolved by Proposition \ref{ERBMGreenDecomp} and \eqref{GreenMatrixGreenFunction}.

We have only given a probabilistic definition of $G_D^{ER}\paren{z,\cdot}$ and our definition is unique only as an element of $L^1\paren{D}$.  It is also possible to give an analytic characterization of $G_D^{ER}\paren{z,\cdot}$.  More specifically, we will prove that there is a version of $G_D^{ER}\paren{z,\cdot}$ that is the unique ER-harmonic function on $D-\set{z}$ satisfying certain boundary conditions (that depend on whether or not $z$ is equal to some $A_i$).  In particular, this will allow us to talk about ``the'' Green's function for ERBM rather than ``a'' Green's function.  We start by proving an analog of \eqref{ERBMPKD} for $G_D^{ER}\paren{z,\cdot}.$

\begin{proposition}\label{ERBMGreenDecomp}
\[G_{D}\paren{z,w}+\sum_{i=1}^n h_i\paren{z}G_{D}^{ER}\paren{A_i,w}\] is a version of $G_D^{ER}\paren{z, \cdot}$.
\end{proposition}

\begin{proof}
This follows easily using the strong Markov property for ERBM and the fact that up until the first time it hits $\partial D$, ERBM has the distribution of a Brownian motion.
\end{proof}

As we expect, $G_D^{ER}\paren{z, \cdot}$ is conformally invariant.  To prove this we need the following lemma, which is a straightforward exercise in measure theory.

\begin{lemma}
If $g\in L^1\paren{D}$, then for all Borel $V \subset D$ we have
\begin{equation*}
\ev{z}{\int_0^{\tau_{D}} \mathbf{1}_{V}\paren{B^{ER}_{D}\paren{t}}g\paren{B^{ER}_{D}\paren{t}} dt}=\int_V G^{ER}_{D}\paren{z,w}g\paren{w} dw.
\end{equation*}
\label{greendefExtension}
\end{lemma}



\begin{proposition}\label{ERBMGreenCI}
If $f: D \rightarrow D'$ is a conformal map, then
\[G^{ER}_{D}\paren{f^{-1}\paren{z},f^{-1}\paren{\cdot}}\]
is a version of $G^{ER}_{D'}\paren{z,\cdot}$.
\end{proposition}

\begin{proof}
It is enough to show that $G^{ER}_{D}\paren{f^{-1}\paren{z},f^{-1}\paren{\cdot}}$ satisfies \eqref{GrDef} for all open subsets of $D'$.  Let $V'$ be an open subset of $D'$ and $V=f^{-1}\paren{V'}$. Using Lemma \ref{greendefExtension} and the change of variables formula, we have
\begin{align}
\int_{V'} G^{ER}_{D}\paren{z,f^{-1}\paren{w}} dw &= \int_{V}G^{ER}_{D}\paren{z,w} \abs{f'\paren{w}}^2 dw\nonumber\\ \label{proofCI2}
&=\ev{z}{\int_0^{\tau_D} \abs{\mathbf{1}_{V}\paren{B^{ER}_{D}\paren{t}}f'\paren{B^{ER}_{D}\paren{t}}}^2 dt}.
\end{align}
Let
\begin{equation}
u\paren{t}=\int_0^t \abs{f'\paren{B^{ER}_{D}\paren{s}}}^2 ds.
\label{proofCI3}
\end{equation}
Substituting $u^{-1}\paren{r}$ for $t$ and using the conformal invariance of ERBM, we see that \eqref{proofCI2} is equal to
\begin{equation}
\ev{z}{\int_0^{\tau_{D'}} \mathbf{1}_{V'}\paren{B^{ER}_{D}\paren{t}} dt},
\label{proofCI4}
\end{equation}
which completes the proof.
\end{proof}
In the proof of Proposition \ref{ERBMGreenCI}, observe that we can only conclude that \eqref{proofCI2} is equal to \eqref{proofCI4} if \eqref{proofCI3} is almost surely finite for all $t<\infty$.  This will be addressed when we prove \eqref{ERBMnolimit}.

In order to prove that $G_D^{ER}\paren{z,\cdot}$ is ER-harmonic, we will need to compute $G_{A_{1,r}}\paren{A_1,\cdot}$.

\begin{lemma}  \label{GreensFunctionAnnulus} Let $A_{1,r} \in \mathcal{Y}_1$ be the annulus with $A_1=\overline{\D}$ and $\partial A_0=\partial B_r\paren{0}$ for some $r>1$ and $B_t$ be a Brownian motion in $r \D$. If $V$ is a Borel set bounded away from $A_1$, then
\begin{equation*}
\ev{A_1}{\int_0^{\tau_{A_{1,r}}} \mathbf{1}_{V}\paren{B_{A_{1,r}}^{ER}\paren{t}} dt}=\ev{0}{\int_0^{\tau_{r\D}} \mathbf{1}_{V}\paren{B_t}}dt,
\end{equation*}
where $\tau_{A_{1,r}}$ and $\tau_{r\D}$ are respectively the first time $B_{A_{1,r}}^{ER}$ leaves $A_{1,r}$ and $B_t$ leaves $r\D$.  Furthermore, we have
\begin{equation*}
G_{A_{1,r}}^{ER}\paren{A_1,z}=\frac{-\log \abs{z}+\log r}{\pi}.
\end{equation*}
\end{lemma}

\begin{proof}
Since $V$ is bounded away from $\overline{D}$, there exists an $\epsilon>0$ such that $V$ is contained in the region bounded by the circle
\[C_{\epsilon}=\set{z \in \C:\abs{z}=1+\epsilon} \]
and the outer boundary of $A_{1,r}$.  Let $\sigma_1=0$, $\tau_j$ be the first time after $\sigma_j$ that $B_{A_{1,r}}^{ER}$ hits $C_{\epsilon}$, and $\sigma_j$ for $j>1$ be the first time after $\tau_{j-1}$ that $B_{A_{1,r}}^{ER}$ hits $A_1$.  Similarly, let $\sigma_1'=0$, $\tau'_j$ be the first time after $\sigma'_j$ that $B_t$ hits $C_{\epsilon}$, and $\sigma'_j$ for $j>1$ be the first time after $\tau'_{j-1}$ that $B_t$ hits the circle of radius $1$, and $\sigma'_1=0$.  It follows from the strong Markov property for ERBM and (3) of Definition \ref{characterizationERBM} that given that $\tau_j<\infty$, the distribution of $B_{A_{1,r}}^{ER}\paren{\tau_j}$ is uniform on $C_{1+\epsilon}$.  It is an easy exercise to check that given that $\tau'_j<\infty$, the distribution of $B_{\tau_j}$ is uniform on $C_{1+\epsilon}$.  Using these two facts, the strong Markov property for ERBM, and the fact that an ERBM has the distribution of a Brownian motion up until the first time it hits the boundary of $A_{1,r}$, we see that
\begin{equation*}
\ev{B_{A_{1,r}}^{ER}\paren{\tau_j}}{\int_{\tau_j}^{\sigma_{j+1}} \mathbf{1}_{V}\paren{B_{A_{1,r}}^{ER}\paren{t}} dt}=\ev{B_{\tau'_j}}{\int_{\tau'_j}^{\sigma'_{j+1}} \mathbf{1}_{V}\paren{B_{r\D}\paren{t}} dt}.
\end{equation*}
Combined with the fact that
\begin{equation*}
\ev{B_{A_{1,r}}^{ER}\paren{\sigma_j}}{\int_{\sigma_j}^{\tau_j} \mathbf{1}_{V}\paren{B_{A_{1,r}}^{ER}\paren{t}} dt}=\ev{B_{\sigma'_j}}{\int_{\sigma'_j}^{\tau'_j} \mathbf{1}_{V}\paren{B_{r\D}\paren{t}} dt}=0,
\end{equation*}
the first result follows.

Using the first part of the proposition, we see that 
\[G_{A_{1,r}}^{ER}\paren{A_1,z}=G_{r\D}\paren{0,z}.\]
Combining this with \eqref{GFUniDisk}, the second part of the proposition follows.
\end{proof}

A quantity that will help us understand $G_D^{ER}\paren{z,\cdot}$ is the density for the amount of time ERBM started at $A_i$ spends in a set from the time it hits a curve $\eta_i$ surrounding $A_i$ until the next time it hits $\partial D$.  The next lemma establishes the existence and some properties of this density.

\begin{lemma} \label{GreenERBMOneTime}
For $i=1, \ldots, n$, let $\eta_i$ and $U_i$ be as in Definition \ref{characterizationERBM} and let $\tau$ be the first time \emph{after} $B_D^{ER}$ has hit one of the $\eta_i$'s that $B_D^{ER}$ hits $\partial D$.  The function
\begin{equation*}
T_i\paren{w}:=G^{ER}_{U_i}\paren{A_i,w}+\int_{\eta_i}G_D\paren{z,w} \frac{H_{\partial U_i}\paren{A_i,z}}{\mathcal{E}_{U_i}\paren{A_i, \eta}} \abs{dz},
\end{equation*}
where by convention we let $G^{ER}_{U_i}\paren{A_i,w}=0$ for $w \notin U_i$, has the following properties.
\begin{enumerate}
\item $T_i\paren{w}$ is a density for the expected amount of time ERBM started at $A_i$ spends in a set up until time $\tau$
\item $T_i\paren{w}$ is harmonic on $D \backslash \eta_i$
\item If $i\neq j$, then $\displaystyle \frac{1}{2}\int_{\eta_j}\frac{d}{dn}~ T_i\paren{w}   \abs{dw}=p_{i j}$, where $n$ is the outward-pointing normal and $p_{i j}$ is as in Section \ref{MarkovChainERBM}
\item If $\eta'_i$ is a smooth curve in the interior of $U_i$ that is homotopic to $\eta_i$, then
\[\frac{1}{2}\int_{\eta'_i}\frac{d}{dn}~ T_i\paren{w}   \abs{dw}=p_{ii}-1.\]
\end{enumerate}
\end{lemma}

\begin{proof}
It is clear using the strong Markov property for ERBM, the fact that ERBM has the distribution of a Brownian motion up until the first time it hits $\partial D$, and (3) of Definition \ref{characterizationERBM} that the first statement holds.

Denote the second summand in the definition of $T_i\paren{w}$ by $S_i\paren{w}$.  If $w \notin \eta_i$ and $\epsilon$ is small enough such that $B\paren{w,\epsilon}$ does not intersect $\eta_i$, then  using Fubini's theorem and the fact that $G_D\paren{z,\cdot}$ is harmonic, we have
\begin{align*}
\frac{1}{2\pi}\int_0^{2\pi} S_i\paren{w+\epsilon e^{i\theta}} d\theta &= \frac{1}{2\pi}\int_0^{2\pi}\bracket{\int_{\eta_i}G_D\paren{z,w+\epsilon e^{i\theta}} \frac{H_{\partial U_i}\paren{A_i,z}}{\mathcal{E}_{U_i}\paren{A_i, \eta}} \abs{dz}} d\theta\\
&=\int_{\eta_i} \frac{H_{\partial U_i}\paren{A_i,z}}{\mathcal{E}_{U_i}\paren{A_i, \eta}}\bracket{ \frac{1}{2\pi}\int_0^{2\pi} G_D\paren{z,w+\epsilon e^{i\theta}} d\theta} \abs{dz}\\
&=\int_{\eta_i} \frac{H_{\partial U_i}\paren{A_i,z}}{\mathcal{E}_{U_i}\paren{A_i, \eta}}G_D\paren{z,w} \abs{dz}\\
&=S_i\paren{w}.
\end{align*}
This shows that $S_i\paren{w}$ satisfies the spherical mean value property at $w$ and, thus, is harmonic on $D \backslash \eta_i$.  It follows that to finish the proof of the second statement, we just have to show that $G^{ER}_{U_i}\paren{A_i,\cdot}$ is harmonic away from $\eta_i$.  Let $f_i:A_{1,r_i} \rightarrow U_i$ be a conformal map mapping the outer boundary of $A_{1,r_i}$ to the outer boundary of $U_i$.  Using Proposition \ref{ERBMGreenCI} and Lemma \ref{GreensFunctionAnnulus}, we see that
\begin{equation}\label{GreenERBMAnnulus}
G^{ER}_{U_i}\paren{A_i,w}=G^{ER}_{D_{r_i}}\paren{A_1,f_i\paren{w}}=\frac{-\log \abs{f_i\paren{w}} + \log r_i}{\pi}.
\end{equation}
Since $\log\abs{z}$ is harmonic and precomposing a harmonic function with a conformal map yields a harmonic function, $G^{ER}_{U_i}\paren{A_i,\cdot}$ is harmonic away from $\eta_i$.

The proof of the third statement uses the fact that if $z$ is in the exterior of $\eta_j$, then
\begin{equation}\label{GreenNormalPoisson1}
\int_{\eta_j}\frac{d}{dn}G_D\paren{z,w} \abs{dw}=2 h_j\paren{z}.
\end{equation}
In the case that $\partial A_j$ is a smooth Jordan curve, this is true because the normal derivative of $G_D\paren{z,w}$ is $2 H_D\paren{z,w}$ on $\partial A_j$ and the integral of the normal derivative of a harmonic function is the same over any two homotopic curves.  If the boundary of $A_j$ is not a smooth Jordan curve, we can map $D$ conformally to a region where the image of $\partial A_j$ is a smooth Jordan curve \cite{MR1344449} and use the conformal invariance of the Green's function, the change of variables formula, the fact that conformal maps preserve angles and the result in the case that $\partial A_j$ is a smooth Jordan curve.  If $i\neq j$, using Fubini's theorem, the dominated convergence theorem, and \eqref{GreenNormalPoisson1}, we have
\begin{align*}
\int_{\eta_j}\frac{d}{dn}~T_i\paren{w} \abs{dw} &= \int_{\eta_j}\frac{d}{dn} \int_{\eta_i}G_D\paren{z,w} \frac{H_{\partial U_i}\paren{A_i,z}}{\mathcal{E}_{U_i}\paren{A_i, \eta}} \abs{dz} \abs{dw}\\
&= \int_{\eta_i}\frac{H_{\partial U_i}\paren{A_i,z}}{\mathcal{E}_{U_i}\paren{A_i, \eta}} \int_{\eta_j}\frac{d}{dn}~G_D\paren{z,w} \abs{dw} \abs{dz}\\
&= 2 \int_{\eta_i}\frac{H_{\partial U_i}\paren{A_i,z}}{\mathcal{E}_{U_i}\paren{A_i, \eta}}h_j\paren{z} \abs{dz}\\
&= 2 p_{i j}
\end{align*}

The proof of the fourth statement is similar to the proof of the third statement and will rely on calculating
$\displaystyle\int_{\eta'_i}\frac{d}{dn}G_D\paren{z,w} \abs{dw}.$  If $z$ is a point in the interior of $\eta'_i$ and $\eta^{''}_i$ is a smooth Jordan curve in the interior of $\eta'_i$ such that $z$ is in the exterior of $\eta^{''}_i$, then by setting up the appropriate contour integral and using Green's theorem, it is not hard to see that (with the normals appropriately oriented)
\[\int_{\eta'_i}\frac{d}{dn}~G_D\paren{z,w} \abs{dw}= \int_{\eta^{''}_i} \frac{d}{dn}~G_D\paren{z,w} \abs{dw} + \int_{B_{\epsilon}\paren{z}} \frac{d}{dn} G_D\paren{z,w} \abs{dw}.\]
Using \eqref{GreenNormalPoisson1} and the fact that $G_D\paren{z,w}=-\frac{\log \abs{z-w}}{\pi}+g_z\paren{w}$, where $g_z$ is harmonic on $D$, we have
\begin{align*}
\int_{\eta'_i}\frac{d}{dn}~G_D\paren{z,w} \abs{dw}&= \int_{\eta^{''}_i} \frac{d}{dn}~G_D\paren{z,w} \abs{dw}+\int_{B_{\epsilon}\paren{z}} \frac{d}{dn}~G_D\paren{z,w} \abs{dw}\\
&=2 h_i\paren{z}-\int_{B\paren{z,\epsilon}} \frac{d}{dn}~\frac{\log \abs{z-w}}{\pi} \abs{dw}\\
&=2 \paren{h_i\paren{z}-1}.
\end{align*}
Using this and arguing as in the proof of the third statement, we have
\begin{align*}
\int_{\eta'_i}\frac{d}{dn}~T_i\paren{w} &= \int_{\eta'_i}\frac{d}{dn} \int_{\eta_i}G_D\paren{z,w} \frac{H_{\partial U_i}\paren{A_i,z}}{\mathcal{E}_{U_i}\paren{A_i, \eta}} \abs{dz} \abs{dw}\\
&= \int_{\eta_i}\frac{H_{\partial U_i}\paren{A_i,z}}{\mathcal{E}_{U_i}\paren{A_i, \eta}} \int_{\eta'_i}\frac{d}{dn}~G_D\paren{z,w} \abs{dw} \abs{dz}\\
&= 2 \int_{\eta_i}\frac{H_{\partial U_i}\paren{A_i,z}}{\mathcal{E}_{U_i}\paren{A_i, \eta}}\paren{h_i\paren{z}-1} \abs{dz}\\
&= 2 \paren{p_{ii}-1}.
\end{align*}
\end{proof}

We have all of the tools necessary to prove that $G_{D}^{ER}\paren{z,\cdot}$ is ER-harmonic.  In what follows, we continue to use the set up of the previous lemma.

\begin{proposition} \label{ERBMGreenERHarmonic}
There are versions of $G_{D}^{ER}\paren{\cdot,z}$ and $G_{D}^{ER}\paren{z,\cdot}$ that are ER-harmonic on $D-\set{z}$.
\end{proposition}
\begin{proof}
Using Proposition \ref{ERBMGreenDecomp}, in order to show that there is a harmonic version of $G_{D}^{ER}\paren{z,\cdot}$, it is enough to show that there is a harmonic version of $G^{ER}_D\paren{A_i,\cdot}$ for each $1 \leq i \leq n$.  Let $\mathbf{T}$ be the vector function with $i$th component $T_i\paren{w}$ and let $\mathbf{D}$ be the diagonal matrix with $ii$ entry $\frac{1}{1-p_{ii}}$.  Using the strong Markov property for ERBM and Lemma \ref{GreenERBMOneTime}, we see that $\mathbf{D}\mathbf{T}$ is the vector function whose $i$th component is the density for the expected amount of time ERBM started at $A_i$ spends in a set up until the first time it hits an $A_j$ with $j\neq i$.  Using \eqref{GreenMatrixFormula} and the strong Markov property for ERBM, we see that the $i$th component of
\begin{equation} \label{GreenMatrixGreenFunction}
\mathbf{D} \mathbf{T}+\mathbf{Q} \mathbf{D} \mathbf{T}+\mathbf{Q}^2 \mathbf{D} \mathbf{T}+ \ldots  = \paren{\mathbf{I}-\mathbf{Q}}^{-1} \mathbf{D} \mathbf{T}
\end{equation}
is a version of $G^{ER}_D\paren{A_i,\cdot}$.  Since $T_i\paren{\cdot}$ is harmonic away from each $\eta_i$, it follows that there is a version of $G_D^{ER}\paren{A_i,\cdot}$ that is harmonic away from each $\eta_i$.  By choosing different $\eta_i$'s and repeating this procedure, we can get a version of $G_D^{ER}\paren{A_i,\cdot}$ that is harmonic away from a sequence of Jordan curves $\eta_1',\ldots, \eta_n'$ which are disjoint from each $\eta_i$.  Finally, since any two versions of $G_D^{ER}\paren{A_i,\cdot}$ are equal almost everywhere, we can find a version of $G_D^{ER}\paren{A_i,\cdot}$ that is harmonic everywhere.

Using (3) and (4) of Lemma \ref{GreenERBMOneTime} and the fact that the $i$th component of \eqref{GreenMatrixGreenFunction} is a version of $G_D^{ER}\paren{A_i,\cdot}$, we see that 
\begin{equation}\label{GreenERBMNormalInt} \int_{\eta_j} \frac{d}{dn}~G^{ER}_D\paren{A_i,w}\abs{dw}= \begin{cases}
0 & \text{ if }  j\neq i \\
-2 &\text{ if }  j=i
\end{cases}.
\end{equation}
It is easy to see using its definition and \eqref{GreenERBMAnnulus} that each $T_i\paren{\cdot}$, and thus each $G_D^{ER}\paren{A_i,\cdot}$, can be extended to a continuous function on $E$.  Combining this with Lemma \ref{ERHarmonicCondition}, we have that that there is a version of $G_D^{ER}\paren{A_i,\cdot}$ that is ER-harmonic on $D-\set{A_i}$.  Finally, using \eqref{GreenNormalPoisson1}, \eqref{GreenERBMNormalInt}, and Lemma \ref{ERHarmonicCondition}, it follows that the version of $G_D^{ER}\paren{z,\cdot}$ defined in Proposition \ref{ERBMGreenDecomp} is ER-harmonic on $D-\set{z}$.

An argument similar to the one showing that $H_D^{ER}\paren{\cdot,z}$ is ER-harmonic shows that $G_D^{ER}\paren{\cdot,z}$ is ER-harmonic.
\end{proof}

We can now give an analytic characterization of $G_D^{ER}\paren{z,\cdot}$.

\begin{proposition}\label{ERBMGreenUnique}
If $z \in D$, then $G_D^{ER}\paren{z,\cdot}$ is the unique ER-harmonic function on $D-\set{z}$ satisfying
\begin{itemize}
\item  $G_D^{ER}\paren{z,w}=-\frac{\log \abs{z-w}}{\pi}+O\paren{1}$ as $w \rightarrow z$
\item  $G_D^{ER}\paren{z,w} \rightarrow 0$ as $w \rightarrow w'$ for any $w' \in \partial A_0$.
\end{itemize}
Furthermore, $G_D^{ER}\paren{A_i,\cdot}$ is the unique ER-harmonic function on $D-\set{A_i}$ that is equal to $G_D^{ER}\paren{A_i,A_i}$ on $\partial A_i$ and $0$ on $\partial A_0$.

\end{proposition}
\begin{proof}
If $z \in D$, the asymptotics for $G_D^{ER}\paren{z,\cdot}$ at the boundary are clear and the asymptotic at $z$ follows from Proposition \ref{ERBMGreenDecomp} and the corresponding result for $G_D\paren{z,\cdot}$.  The uniqueness follows from a proof similar to the corresponding result for $G_D\paren{z,\cdot}$ (see \cite{MR2129588}, pg. 54).  The second statement follows from an extension of Proposition \ref{ERHarmonicFromERBM}.
\end{proof}

In what follows, when we write $G_D^{ER}\paren{z,\cdot}$ or $G_D^{ER}\paren{\cdot,w}$ we will mean a version that is ER-harmonic.

\begin{corollary}\label{GreenSym}
$G_D^{ER}\paren{z,w}=G_D^{ER}\paren{w,z}$ for all $z, w \in E$.
\end{corollary}
\begin{proof}
$G_D^{ER}\paren{\cdot,z}$ satisfies the conditions of Proposition \ref{ERBMGreenUnique} and thus, is the same function as $G_D^{ER}\paren{z,\cdot}$.
\end{proof}

\subsection{Proofs of formulas (\ref{ERBMalltime}) and (\ref{ERBMnolimit})} \label{sectGFPOF}
\label{chapGFPOF}
The theory of Green's functions for ERBM can be used to prove formulas \eqref{ERBMalltime} and \eqref{ERBMnolimit}.  We start with a lemma.

\begin{lemma}
Let $A_{1,r} \in \mathcal{Y}_1$ be as in Lemma \ref{GreensFunctionAnnulus} and $\tau=\inf\set{t:B^{ER}_{A_{1,r}}\paren{t} \in \partial A_0}$.  If $f: A_{1,r} \rightarrow D$ is a conformal map and $D$ is bounded, then
\[\ev{z}{\int_0^{\tau} \abs{f'\paren{B^{ER}_{A_{1,r}}\paren{s}}}^2 ds} < \infty.\]
\label{allpathlemma}
\end{lemma}

\begin{proof}
Using Lemma \ref{greendefExtension}, we have that for sufficiently small $\epsilon$
\begin{align*}
\ev{}{\int_0^{\tau} \abs{f'\paren{B^{ER}_{A_{1,r}}\paren{s}}}^2 ds}=&\int_{A_{1,r}}G^{ER}_{A_{1,r}}\paren{z,w}\abs{f'\paren{w}}^2 dw\\
=&\int_{B_{\epsilon}\paren{z}}G^{ER}_{A_{1,r}}\paren{z,w}\abs{f'\paren{w}}^2 dw\\
&+\int_{A_{1,r} \backslash {B_{\epsilon}\paren{z}}}G^{ER}_{A_{1,r}}\paren{z,w}\abs{f'\paren{w}}^2 dw.
\end{align*}
Since $\abs{f'\paren{r}}$ is bounded on ${B_{\epsilon}\paren{z}}$ and the Green's function for ERBM is integrable, the first integral in the sum is finite.  Since $G^{ER}_{A_{1,r}}\paren{z,\cdot}$ is bounded on $A_{1,r} \backslash B\paren{z,\epsilon}$ and $\int_{A_{1,r}} \abs{f'\paren{w}}^2 dw$ is equal to the area of $D$ (by a straightforward change of variables), the second integral in the sum is also bounded.  
\end{proof}

\begin{proposition} \label{allpaththeoremERBM}
Let $f: \C\backslash\D \rightarrow D$ be a conformal map sending $\infty$ to $\infty$. Then a.s. we have
\begin{equation}
\int_0^t \abs{f'\paren{B_{\C \backslash \D}^{ER}\paren{s}}}^2 ds < \infty
\label{nolimitproofERBM}
\end{equation}
and
\begin{equation}
\int_0^{\infty} \abs{f'\paren{B_{\C \backslash \D}^{ER}\paren{s}}}^2ds = \infty.
\label{allpathproofERBM}
\end{equation}
\end{proposition}

\begin{proof}
For fixed $t$, let $W$ be the set of $\omega$ in the underlying probability space such that the left hand side of \eqref{nolimitproofERBM} is infinite and for each $n \in \N$, let $W_n$ be the set of $\omega$ such that $B^{ER}_{D}$ has not left $A_{1,n}$ by time $t$.  By Lemma \ref{allpathlemma}, the measure of $W_n \cap W$ is zero.  It follows that for almost every $\omega \in W$, the path of $B^{ER}_{\C\backslash \D}$ up to time $t$ is unbounded.  It is easy to see from the definition of ERBM that this implies that $W$ has measure $0$.

It is easy to see that $\abs{f'}$ is bounded below on the set
\[\set{z \in \C: \abs{z}>2}.\]
Since the set of $t$ such that $\abs{B^{ER}_{\C\backslash\D}\paren{t}}>2$ has infinite measure, \eqref{allpathproofERBM} follows.
\end{proof}

Proposition \ref{allpaththeoremERBM} clarifies the implicit use of \eqref{nolimitproofERBM} and its analogs.  The reader can verify that the proof of Proposition \ref{allpaththeoremERBM} does not rely on any of the results that used \eqref{nolimitproofERBM}.  For instance, in the proof of Proposition \ref{ERBMGreenCI} we used the fact that a.s. \eqref{nolimitproofERBM} holds for any finitely connected region $D$. Using the definition of ERBM, it is easy to see that to prove this, it is enough to prove it for any domain conformally equivalent to $\C\backslash \D$.  Notice, however, that the only property of $\C\backslash \D$ we used in the proof of Lemma \ref{allpathlemma} was that $G_{A_{1,r}}^{ER}\paren{z,\cdot}$ is bounded away from $z$ and integrable in a neighborhood of $z$.  Once we know Proposition \ref{allpaththeoremERBM} holds, the proof of Proposition \ref{ERBMGreenCI} works for $D=\C\backslash \D$ and we can use Proposition \ref{ERBMGreenCI} and Proposition  \ref{ERBMGreenDecomp} to conclude that $G_{D}^{ER}\paren{z,\cdot}$ is bounded away from $z$ and integrable in a neighborhood of $z$ for any region $D$ conformally equivalent to $\C \backslash \D$.  This allows us to prove an analog of Proposition \ref{allpaththeoremERBM} for any conformal annulus, which is what we needed.

\subsection{Conformal Mapping Using $G_D^{ER}(z,\cdot)$}

We call a domain a \emph{bilateral standard domain} if it is an annulus of outer radius $1$ with a finite number of concentric arcs removed.  We call a domain a \emph{standard domain} if it is the unit disk with a finite number of concentric arcs removed.  Analogous to the connection between $H_D^{ER}\paren{\cdot,w}$ and conformal maps onto chordal standard domains, there is a connection between $G_D^{ER}\paren{z,\cdot}$ and conformal maps onto bilateral standard domains and standard domains.  We start by giving analogs of Lemma \ref{PKLSConnected}, Proposition \ref{PKFullRank}, and Proposition \ref{PKLevelCurveJordan} for $G_D^{ER}$.  The proofs are similar and are omitted.

\begin{lemma} \label{GFLSConnected}
For each positive real number $r$, the set
\[V_r=\set{w:G_D^{ER}\paren{z,w} \leq r} \]
is connected. Furthermore, for each $z \in V_r$, there is a path contained in $V_r$ starting at $z$ and ending at a point in $\partial A_0$.
\end{lemma}

\begin{proposition}\label{GFFullRank}
For all $w \in D$, the derivative of $G_D^{ER}\paren{z,\cdot}$ at $w$ has full rank.
\end{proposition}

\begin{proposition} \label{GFLevelCurveJordan}
If $r$ is a positive real number, then
\[ \set{w:G^{ER}_D\paren{z,w}=r} \]
is a Jordan curve $\gamma_r$.  Furthermore, $\gamma_r$ separates $D$ into two connected components
\[ \set{w: G^{ER}_D\paren{z,w}<r} \text{ and } \set{w: G^{ER}_D\paren{z,w}>r }.\]
\end{proposition}

It is a classical theorem of complex analysis \cite{MR510197} that any $D \in \mathcal{Y}$ is conformally equivalent to a bilateral standard domain.  Using $G_D^{ER}\paren{A_i,\cdot}$, we can give a new proof of this fact.

\begin{theorem}\label{bilateralGreenERBM}
Let $D \in \mathcal{Y}_n$ and suppose that there is no Jordan curve in $D$ with $A_0$ in its interior.  If $u=\pi G_D^{ER}\paren{A_i,\cdot}$, then there is a bilateral standard domain $D'$ and a conformal map $f=e^{-\paren{u+iv}}$ from $D$ onto $D'$.  Furthermore, if $g$ is another conformal map from $D$ onto a bilateral standard domain $D''$ and $g$ maps $\partial A_i$ onto the inner radius of $D''$ and $\partial A_0$ onto the outer radius of $D''$, then $f$ and $g$ differ by a rotation.
\end{theorem}

\begin{proof}
Fix $z_0\in D$ and let $v\paren{z_0}=x \in \R$ and
\begin{equation}\label{BLFunctionDef}
v\paren{z}=v\paren{z_0} + \int_{\gamma} \frac{d}{dn}~ u\paren{z} \abs{dz},
\end{equation}
where $\gamma$ is a smooth curve connecting $z_0$ and $z$ and $n$ is a normal vector.  It follows from the Cauchy-Riemann equations that $u+iv$ is locally holomorphic.  Using \eqref{GreenERBMNormalInt}, we see  that $v$ is well-defined up to an integer multiple of $2\pi.$  It follows that $f=e^{-\paren{u+iv}}$ is a well-defined holomorphic function on $D$.

By an extension of the maximal principle for ERBM, $u$ attains its maximum on $\partial A_i$.  Using this, we see that the image of $f$ is contained in the annulus of inner radius $e^{-u\paren{A_i}}$ and outer radius $1$.  By Proposition \ref{GFLevelCurveJordan}, $\set{z \in D:u\paren{z}=r}$ is a Jordan curve $\gamma_r$.  By \eqref{GreenERBMNormalInt},
\[ \int_{\gamma_r} \frac{d}{dn}~ u\paren{z} \abs{dz}=-2\pi.\]
It follows that if $r \neq u\paren{A_j}$ for any $j$, then $f$ maps $\gamma_r$ injectively onto the circle of radius $e^{-r}$ and if $r=u\paren{A_j}$ for some (or several) $j$, then $f$ maps $\gamma_r$ injectively onto the circle of radius $e^{-r}$ with one (or several) arc(s) removed.  Putting all of this together, we see that $f$ is a conformal map onto a bilateral standard domain.

Suppose $g=e^{-\paren{u+iv}}$ is another conformal map from $D$ onto a bilateral standard domain $D''$ and $g$ maps $\partial A_i$ onto the inner radius of $D''$ and $\partial A_0$ onto the outer radius of $D''$.  To prove the uniqueness statement of the theorem, it is enough to show that $u=\pi G_D^{ER}\paren{A_i,\cdot}$.  Observe that $-\log\paren{g}$ is a locally holomorphic, multi-valued function well-defined up to an integer multiple of $2\pi i$.  As a result, $u$ is a well-defined harmonic function.  Let $\eta_j$ for $j\neq i$ be a Jordan curve surrounding $A_j$ whose interior contains no point of $A_k$ for $j\neq k$.  On the interior of $\eta_j$, $u+iv$ is a well-defined holomorphic map and as a result,
\[\int_{\eta_j'} \frac{d}{dn}~ u\paren{z} \abs{dz}=0\]
for any Jordan curve $\eta_j'$ surrounding $A_j$ and in the interior of $\eta_j$.  We conclude by Lemma \ref{ERHarmonicCondition} that $u$ is ER-harmonic on $D \backslash {A_i}$ and since it is equal to zero on $\partial A_0$, it must be a multiple of $G_D^{ER}\paren{A_i,\cdot}$.  Using \eqref{GreenERBMNormalInt}, it is easy to see that the only multiple that will work is $\pi$.
\end{proof}

Using $G_D^{ER}\paren{z,\cdot}$ instead of $G_D^{ER}\paren{A_i,\cdot}$, we can prove another classical conformal mapping theorem.

\begin{theorem}
Let $D \in \mathcal{Y}_n$ and suppose that there is no Jordan curve in $D$ with $A_0$ in its interior.  If $u=\pi G_D^{ER}\paren{z,\cdot}$, then there is a standard domain $D'$ and a conformal map $f=e^{-\paren{u+iv}}$ from $D$ onto $D'$.  Furthermore, if $g$ is another conformal map from $D$ onto a bilateral standard domain that sends $z$ to $0$, then $f$ and $g$ differ by a rotation.
\end{theorem}

\begin{proof}
The proof is similar to that of Theorem \ref{bilateralGreenERBM} and is omitted.
\end{proof}




\bibliography{paper}
\bibliographystyle{amsplain}
\end{document}